\documentclass[11 pt,a4paper]{article}

\usepackage{amscd}
\usepackage{amsfonts,amsmath,amssymb,amsthm}
\usepackage{array}
\usepackage{enumerate}
\usepackage{graphicx}
\usepackage[all]{xy}
\usepackage[colorlinks=true]{hyperref}
\usepackage{fullpage}
\usepackage{tikz}

\newtheorem{thm}{Theorem}[section]

\newtheorem*{thm-prim}{Theorem~\ref{thm:prim}}
\newtheorem*{thm-nprim}{Theorem~\ref{thm:nprim}}

\newtheorem{lem}[thm]{Lemma}
\newtheorem{prop}[thm]{Proposition}
\newtheorem*{prop:prim->min}{Proposition~\ref{prop:prim->min}}

\theoremstyle{definition}
\newtheorem{defn}[thm]{Definition}
\newtheorem{example}[thm]{Example}

\newtheorem{rem}[thm]{Remark}

\newcommand{\C}{\mathbb{C}}
\newcommand{\N}{\mathbb{N}}
\newcommand{\R}{\mathbb{R}}

\newcommand{\Z}{\mathbb{Z}}

\newcommand{\M}{\mathbb{M}}
\newcommand{\E}{\mathbb{E}}
\newcommand{\Hyp}{\mathbb{H}}

\newcommand{\B}{\mathsf{B}}
\newcommand{\supp}{\text{supp}}

\newcommand{\im}{\text{Im}}

\newcommand{\val}{\text{val}}

\newcommand{\core}{\text{Core}}
\newcommand{\T}{\mathcal{T}}
\newcommand{\diam}{\text{diam}}

\newcommand{\A}{\mathcal{A}}
\newcommand{\type}{\text{type}}
\newcommand{\isom}{\text{Isom}}

\renewcommand {\epsilon}{\varepsilon}
\renewcommand {\phi}{\varphi}

\renewcommand {\leq}{\leqslant}
\renewcommand {\geq}{\geqslant}

\def\strutdepth{\dp\strutbox}
\def \ss{\strut\vadjust{\kern-\strutdepth \sss}}
\def \sss{\vtop to \strutdepth{
\baselineskip\strutdepth\vss\llap{$\diamondsuit\;\;$}\null}}

\def\strutdepth{\dp\strutbox}
\def \sst{\strut\vadjust{\kern-\strutdepth \ssss}}
\def \ssss{\vtop to \strutdepth{
\baselineskip\strutdepth\vss\llap{$\spadesuit\;\;$}\null}}

\def\strutdepth{\dp\strutbox}
\def \ssh{\strut\vadjust{\kern-\strutdepth \sssh}}
\def \sssh{\vtop to \strutdepth{
\baselineskip\strutdepth\vss\llap{$\heartsuit\;\;$}\null}}

\begin{document}
\title{Geometric realizations of $2$-dimensional substitutive tilings}
\date{}

\author{Nicolas B\'edaride \footnote{LATP  UMR 7353, Universit\'e Paul C\'ezanne, 
avenue escadrille Normandie Niemen, 
13397 Marseille cedex 20, France.
{\it E-mail address:} {\tt nicolas.bedaride@univ-amu.fr}
}
\& Arnaud Hilion
\footnote{LATP  UMR 7353, Universit\'e Paul C\'ezanne, avenue escadrille Normandie Niemen, 
13397 Marseille cedex 20, France.
{\it E-mail address:} {\tt arnaud.hilion@univ-amu.fr}}
}

\maketitle


\begin{abstract}
We define 2-dimensional topological substitutions.
A tiling of the Euclidean plane, or of the hyperbolic
plane, is substitutive if the underlying 2-complex
can be obtained by iteration of a 2-dimensional 
topological substitution.
We prove that there is no primitive substitutive tiling of 
the hyperbolic plane $\mathbb{H}^2$.  
However, we give an example of a substitutive tiling of 
$\Hyp^2$ which is non-primitive.  
\end{abstract}

\tableofcontents

\section{Introduction}

Let $\M$ be the Euclidean plane $\E^2$ or the hyperbolic
plane $\Hyp^2$, with a given base point $O$ called the 
origin of $\M$.
The space of tilings \cite{Bel.Ben.Gam.06,Sad.08}
defined by a(typically finite) set of tiles $\textsf{T}$ of $\M$
is the set of all the tilings of $\M$ that can be 
obtained using the translates 
of $\textsf{T}$ by a given subgroup $\Gamma$
of the group $\isom(\M)$ of isometries of $\M$.
The group $\Gamma$ naturally acts on this set.
This set is given a metric topology which says that two tilings 
$x_1$ and $x_2$ are close to each other if there exists a ball 
$B$ of large radius centered at the origin $O$ of $\M$ and 
an element $g\in\isom(\M)$ which  does 
not move $B$ much, 
such that the restrictions of $x_1$, $g x_2$ 
to $B$ are equal. Details are given in Section~\ref{spacetiling}; 
in particular Proposition \ref{lem:distance} gives an explicit
distance on a space of tilings.

A space of tilings, with this topology and the action of 
$\Gamma$, can be viewed as a dynamical system.
If $x$ is a tiling in a given space of tilings, $\Omega(x)$
will denote the hull of $x$, which is the closure of the
$\Gamma$-orbit of $x$ in the space of tilings.

A finite connected union of tiles in $x$ is called a patch of the tiling $x$.
A tiling $x$ of $\M$ is repetitive if for any patch $P$ of $x$, 
there exists some $r>0$ such that in any ball of $\M$ of
radius $r$, one can see the translate 
$gP$ of $P$ by some element
$g\in\Gamma$.
The minimality of the dynamical
system $(\Omega(x),\Gamma)$ 
(i.e. the fact that every $\Gamma$-orbit is dense in $\Omega(x)$)
is related to the property
of repetitivity of the tiling $x$
by the following proposition, which is the analogue in dimension $2$ 
of a rather well known fact in dimension 1 (i.e. for subshifts). In dimension $2$ it 
is known in several cases (in particular, if 
$\M=\E^2$ and $\Gamma$ is the set of translations of $\E^2$).
We derive below (see Proposition \ref{prop:repet}) the statement in a context adapted to our purpose.

\begin{prop}[Gottschalk's Theorem]
Let $x$ be a tiling, and let 
$\Omega(x)$ be its hull.
\begin{enumerate}[(i)]
\item  If the tiling $x$ is repetitive, then the dynamical system 
          $(\Omega(x),\Gamma)$ is minimal.
\item If the dynamical system $(\Omega(x),\Gamma)$ is minimal, 
         and $\Omega(x)$ is a compact set, then the tiling $x$ is repetitive.
\end{enumerate}
\end{prop}

Examples of repetitive tilings of $\E^2$ are given by periodic
tilings, i.e. tilings invariant under 
some group isomorphic to $\Z^2$
acting discretely and cocompactly on $\E^2$ by translations. 
The hull of such a tiling is homeomorphic to a torus.
More interesting examples are given by aperiodic tilings (which
are not periodic). This is the case, for instance, 
for 
the famous
tiling given by Penrose in the early 70's, \cite{Penr.84}.

Since then, a lot of examples of such aperiodic tilings of $\E^2$ 
have been produced, based on several recipes at our disposal: 
the ``cut and project'' method for instance,
but of more interest regarding the subject of this article,
the ``substitutive examples".
The existing terminology concerning the latter is not very well established: 
we have to accomodate to 
several alternative definitions, all having
their 
own interest, see \cite{Priebe.08, Robin.04} for detailed discussion.

These substitutive tilings share the fact that their
construction involves 
a 
homothety $H$ of $\E^2$
of coefficient $\lambda>1$. 
Typically, the situation is 
as follows:
We have a finite number of polygons 
$T_1,\dots,T_d$ of $\E^2$,
such that their images by the homothety $H$ can
be tiled by translated images of $T_1,\dots,T_d$ of $\E^2$
(see Figure \ref{CHAISE} for a famous example).

\begin{figure}[ht] 
  \begin{center} 

  \includegraphics[width=8cm]{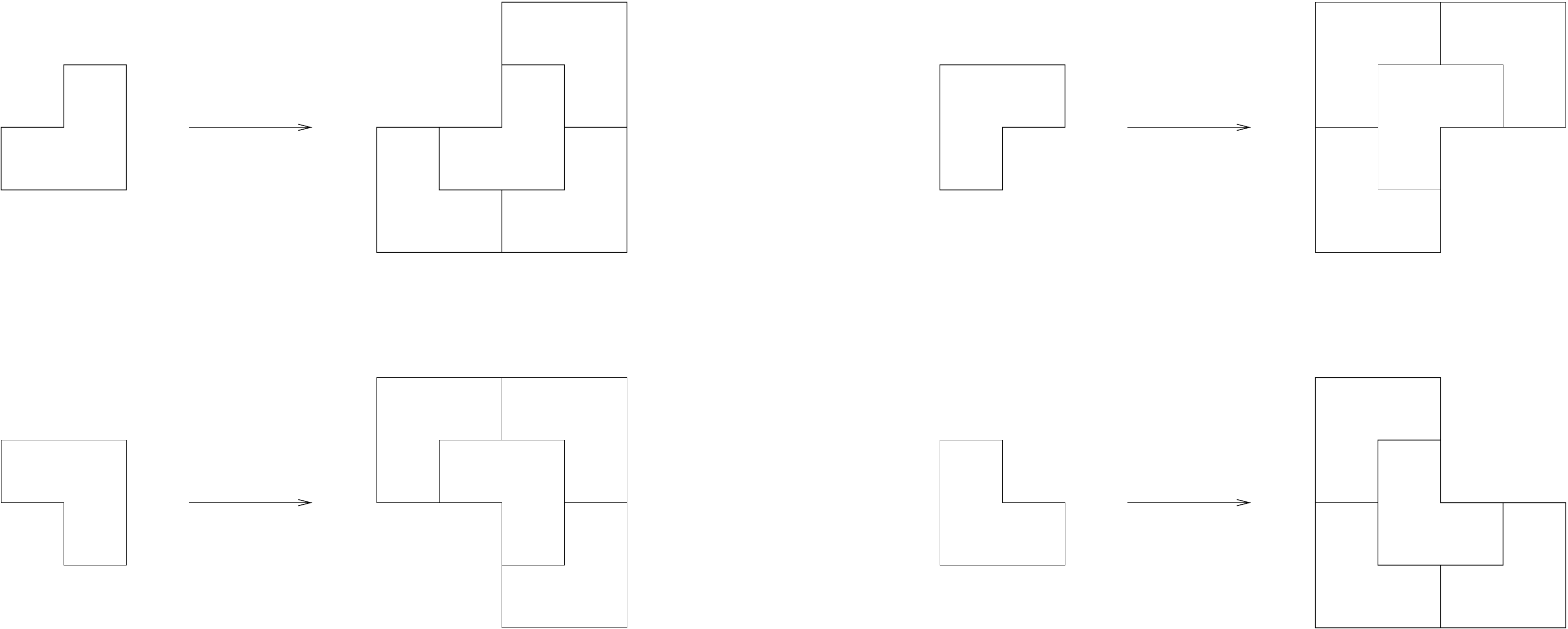}
  \end{center}
  \caption{The ``chair substitution'' is defined on 4 tiles
  $T_1$, $T_2$, $T_3$ and $T_4$, each of which differs 
  from the previous one by a rotation of angle $\frac{\pi}{2}$. 
  The image of a tile $T_i$ by
  the homothety of coefficient 2 is tiled by 4 translated copies of
  the  $T_i$.}
  \label{CHAISE}
\end{figure}

Iteration of this process leads to bigger and bigger
parts $H^k(T_1)$ of the plane tiled by translated images 
of $T_1,\dots,T_d$
(see Figure \ref{chaise-iteree}), and, up to extraction of a subsequence,
we obtain a tiling of the whole plane $\E^2$.

\begin{figure}[ht]
  \begin{center} 
 \includegraphics[width= 6cm]{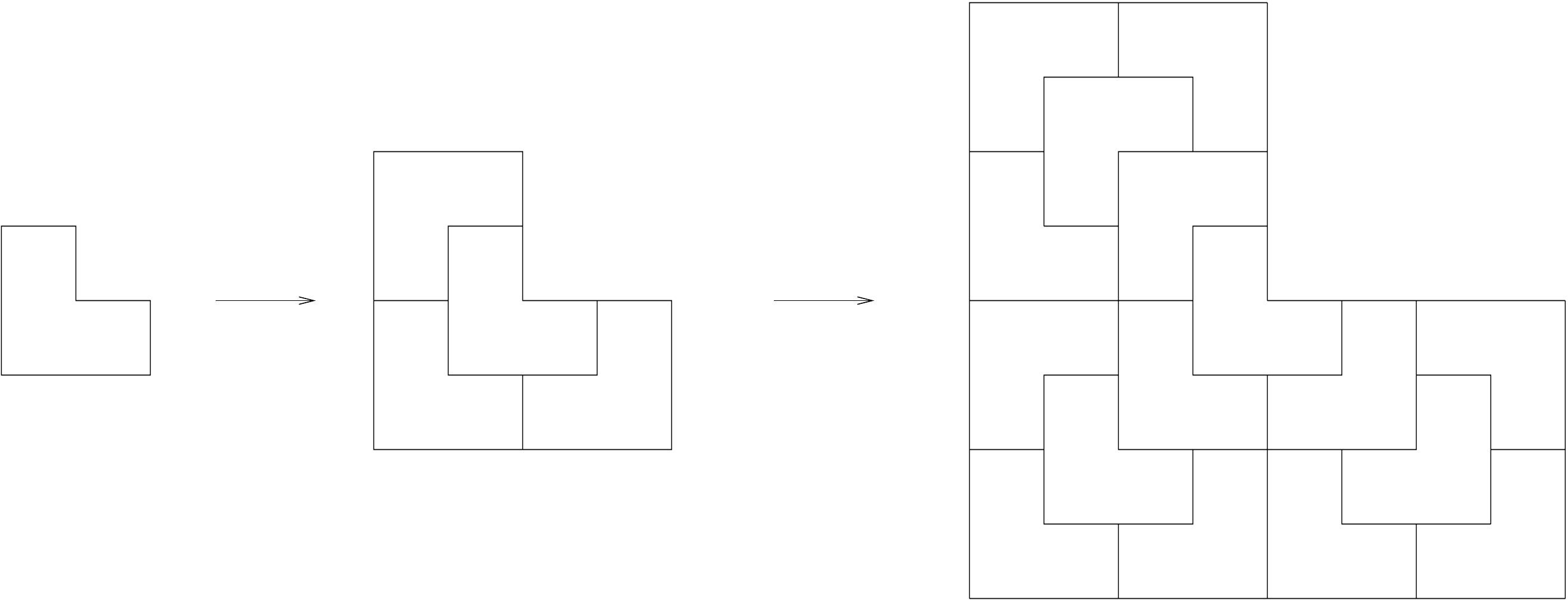}
  \end{center}
  \caption{When iterating the ``chair substitution'' on a tile,
  bigger and bigger parts of the plane are covered by tiles.}
  \label{chaise-iteree}
\end{figure}

Underlying these constructions, there are topological 
(or combinatorial) objects. Indeed, if we forget the metric
data, there 
remains a finite set of 2-cells $T_1,\dots,T_d$
and a map $\sigma$ which associates to any of these cells, say $T_i$,
a patch $\sigma(T_i)$, homeomorphic to a disc, obtained 
by gluing some copies of the $T_j$ along edges.
Moreover, such a map $\sigma$ can be consistently 
iterated. This is what we would like to call a
2-dimensional topological substitution:
for a formal definition
see Sections~\ref{sec:top} and \ref{sec:top-subst}, 
and 
in particular Definition \ref{def:top-sub}.
By iterating a topological substitution, we obtain
a 2-complex
$X$ homeomorphic to the plane,
each 2-cell of which is labelled
by some $T_i$.

A tiling $x$ of $\M$ is substitutive if the underlying 2-complex
can be obtained by iterating a 2-dimensional
topological substitution $\sigma$.
Under some 
primitivity hypothesis (see Definition \ref{def:pureprim}) 
on the substitution $\sigma$, the hull of $x$
is minimal:

\begin{prop:prim->min}
If $x$ is a primitive substitutive tiling of $\M$, 
then the hull $(\Omega(x),\Gamma)$ is minimal.
\end{prop:prim->min}

The idea of defining inflation processes on topological complexes is not new. 
For instance, it has been developed by Previte \cite{Prev-98} or 
Priebe Frank \cite{Frank-03} in dimension $1$ ({\it i.e.} for graphs). 
The version by Priebe Frank which focuses in particular on planar 
graphs, can be seen as a kind of dual construction of our topological substitution.

\medskip

The main question we investigate in this 
paper 
is:
{\em Does there exist 
a subsitutive tiling of the hyperbolic 
plane $\Hyp^2$ ?}

\medskip

The fact that there is no homothety
of ratio $\lambda>1$ in $\Hyp^2$ leads us to 
reformulate
the question as 
follows: 
Given a topological substitution $\sigma$ and a 2-complex
$X$ homeomorphic to the plane obtained by iteration of
$\sigma$, is it possible to geometrize $X$ as a tiling of
$\Hyp^2$ ? 
That is, 
can one 
associate 
to each cell $T_i$ of the
substitution a polygon $\lfloor T_i \rfloor$ of $\Hyp^2$
such that the resulting metric 
space, 
obtained by gluing
the polygons as prescribed by $X$, is isometric to the 
hyperbolic plane $\Hyp^2$ ?
We prove:

\begin{thm-prim}
There does not exist 
a 
primitive substitutive tiling of the 
hyperbolic plane $\Hyp^2$. 
\end{thm-prim}

Thus, to obtain examples of aperiodic tilings of the
hyperbolic plane $\Hyp^2$, one can not simply use, as in 
the Euclidean plane $\E^2$, a primitive substitution
(compare Proposition \ref{prop:prim->min}).
However, one can exhibit a strongly aperiodic set of tiles,
(i.e. a finite set of tiles in $\Hyp^2$ 
such that $\Hyp^2$ can be tiled by their isometric images, but no such tiling of $\Hyp^2$ is invariant under any
isometry $g\in\isom(\Hyp^2)$, $g\neq Id$).
The first example of such a strongly aperiodic set of tiles
of $\Hyp^2$ was 
proposed by Goodman-Strauss \cite{Good.05},
and is based on ``regular production 
systems'', 
which are quite elaborated objects.

Given a primitive substitution $\sigma$ and a tile $T$, 
the proof of Theorem~\ref{thm:prim}
relies on an estimate of the growth rates of  
both
the number of tiles in $\sigma^n(T)$ and the number of 
tiles in the boundary of $\sigma^n(T)$. These growth rates 
contradict the linear isoperimetric inequality which 
characterizes hyperbolicity. 

Nevertheless, we obtain:

\begin{thm-nprim}
There exists a (non-primitive) substitutive tiling of the 
hyperbolic plane $\Hyp^2$. 
\end{thm-nprim}

This theorem is proved by giving in Section~\ref{subsec:subst-pav} an explicit example of 
a non-primitive 
topological substitution. In the 
unlabelled 
complex obtained 
by iterating $\sigma$, each vertex is of valence three and 
each face is  
a heptagon. 
Thus it can be geometrized as the regular tiling of $\Hyp^2$ 
by regular heptagons of angle $\frac{2\pi}{3}$.
The trick of how 
to define a convenient substitution consists 
of 
using 
a tile $T$ as a source: at each iteration, $T$ produces an annulus 
surrounding $T$. Each tile in the annulus grows as an 
unidimensional complex, in such a way that the 
$n$-th image of the 
annulus is still an annulus perfectly surrounding the 
$(n-1)$-st
image, the whole picture forming a kind of ``target board".

The construction is a little bit tricky, and
we do not know, in general, which regular tilings of $\Hyp^2$ can be 
obtained in such a way. For instance, can we obtain the tiling by regular
squares of angle $\frac{2\pi}{5}$?

\medskip

The paper is organized as follows.
Section \ref{sec:tiling} gathers basic notions about tilings.
We then focus on spaces of tilings in Section \ref{sec:dynamics}.
Section \ref{sec:top} is devoted to the definition of 
a topological 
pre-substitution: we make explicit the condition, called compatibility, 
which 
a pre-substitution
must satisfy 
so that one can iterate it consistently.
Then we define topological 
substitutions 
in Section \ref{sec:top-subst}, 
and explain what a substitutive tiling is. 
Primitive topological 
substitutions 
are studied in Section \ref{sec:prim}, 
where we prove that primitivity implies minimality. 
We give the proof of Theorem \ref{thm:prim} in Section \ref{sec:hyp}. 
In Section~\ref{sec:bound} we explain how to check that the valence 
of the vertices in the complex obtained from a given topological 
substitution is uniformly bounded, which is a necessary condition on the complex
for it to be geometrizable.
Finally we give an example of a substitutive tiling of the hyperbolic 
plane in Section~\ref{sec:subst-pav}, 
thus proving Theorem~\ref{thm:nprim}. We give in an appendix proofs of propositions stated in Section~\ref{sec:dynamics}. 

\medskip

{\em Acknowledgment: We thank Martin Lustig for remarks and comments on this paper. This work has been supported by the Agence Nationale de la Recherche -- ANR-10-JCJC 01010.}


\section{Tilings}\label{sec:tiling}
In this section we generalize standard notions on tiling to a general metric space. 
For the references about this material we refer the reader to \cite{Robin.04}.
\subsection{Background}\label{sec:back}
Let $(M,d)$ be a complete metric space, equipped with a distinguished point $O$, called the origin. We denote by $\isom (M)$ the group of isometries of 
$M$,
and by $\Gamma$ a subgroup of $\isom(M)$.

A {\bf tile} is a compact subset of $M$ which is the 
closure of its interior
(in most of the basic examples, a tile is homeomorphic to a closed ball). 
We denote by $\partial{T}$ the boundary of a tile, 
i.e. $\partial{T}=T\smallsetminus \mathring{T}$.
Let $\A$ be a finite set of {\bf labels}. 
A {\bf labelled tile} is a pair $(T,a)$ where
$T$ is a tile and $a$ an element of $\A$.

Two labelled tiles $(T,a)$ and $(T',a')$ are {\bf equivalent} if 
$a=a'$ and there exists an isometry $g\in\Gamma$ 
such that $T'=gT$. 
An equivalence class of labelled tiles is called a {\bf prototile}.
We stress that a prototile depends on the choice of the subgroup
$\Gamma$ of $\isom(M)$.
In many cases, one does not need the labelling 
to distinguish different prototiles, for example if we consider a 
family of prototiles such that the tiles in two different prototiles 
are not isometric.

\begin{defn}
A tiling $(M,\Gamma,\mathcal{P},\textsf{T})$ of the space $(M,d)$ 
modelled on a set of prototiles $\mathcal{P}$,
is a set $\textsf{T}$ of tiles, each belonging to an element of 
$\mathcal{P}$, such that:
\begin{enumerate}[(i)]
\item $M=\displaystyle\bigcup_{T\in\textsf{T}}T,$
\item two distinct tiles of $\textsf{T}$ have disjoint interiors.
\end{enumerate}
\end{defn}

A connected finite union (of labelled) tiles is called a (labelled) {\bf patch}. 
Two finite patches are 
{\bf equivalent} if they have the same number $k$ of tiles and these 
tiles can be indexed $T_1,\dots,T_k$ and $T'_1,\dots,T'_k$, 
such that there exists $g\in 
\Gamma$ with $T'_i=g T_i$ for all $i\in\{1,\dots,k\}$. 
Two labelled patches are {equivalent} if moreover $T_i,T'_i$ 
have same labelling for all $i\in \{1,\dots,k\}$.
An equivalence class of patches is called a {\bf protopatch}.

The {\bf support} of a patch $P$, denoted by $\supp(P)$, 
is the subset of $M$ which
consists of points belonging to a tile of $P$.
A {\bf subpatch} of a patch $P$ is a patch which is a subset of the patch $P$.

Let $x=(M,\Gamma,\mathcal{P},\textsf{T})$ be a tiling, and
let $A\subseteq  M$ be a subset of $M$.
A patch $P$ {\bf occurs in $A$} if there exists 
$g\in\Gamma$ such that for any tile $ T\in P$,
$gT$ is a tile of $\textsf{T}$ which is contained in $A$:
$$gT\in \textsf{T}\;\; \text{ and }  \;\;\supp(gT)\subseteq A.$$
We note that any patch in the protopatch $\hat{P}$ defined by $P$ occurs in $A$. We says that the protopatch $\hat{P}$ {\bf occurs} in $A$.\\
The {\bf language} of $x$, denoted $\mathcal{L}(x)$, is the set of protopatches of $x$.

If $M=\mathbb{M}$, the euclidean plane or the hyperbolic plane, a polygon in $M$ is a (not necessarily convex) compact subset homeomorphic to a ball
and such that its boundary is a finite union of geodesic segments in 
$\mathbb{M}$. When all tiles of $\mathcal{P}$ are polygons in $\mathbb{M}$, the 
tiling is called a {\bf polygonal tiling}.

\subsection{The $2$-complex defined by a tiling of the plane}\label{section:complex}
For this part let us assume $M=\mathbb{M}$.
Let $X$ be a $n$-dimensional CW-complex 
(see, for instance, \cite{Hat.02} for basic facts about CW-complexes). 
We denote by $|X|$ the number of $n$-cells in $X$. The subcomplex 
of $X$ which consists of cells of dimension at most $k\in\{0,\dots,n\}$ is denoted by $X^k$.
In this paper, we will only consider CW-complexes of dimension
at most 2. The 0-cells will be called vertices, the 1-cells edges
and the 2-cells faces. 

Let $(\mathbb{M},\Gamma,\mathcal{P},\textsf{T})$ be a tiling of the plane. 
We suppose that the tiles are connected.
This tiling defines naturally a $2$-complex 
$X=X_{(\mathbb{M},\Gamma,\mathcal{P},\textsf{T})}$ in the following way. The set $X^0$ of vertices of $X$ is the set of points in $\mathbb{M}$ which belong to 
(at least) three tiles of $\mathcal{T}$. Each connected component of the set 
$\displaystyle\bigcup_{T\in\textsf{T}}\partial{T}\setminus X^0$ 
is an open arc. Any closed edge of $X$ is the closure of one of these arcs 
(and conversily).
Such an edge $e$ is glued to the endpoints $x,y\in X^0$
of the arc. The set of faces of $X$ is the set of
tiles of $(\mathbb{M},\Gamma,\mathcal{P},\textsf{T})$.
We remark that the boundary of a tile is a subcomplex of $X^1$ homeomorphic 
to the circle $\mathbb{S}^1$: this gives the gluing of the corresponding face on the 
1-skeleton. If $(\mathbb{M},\Gamma,\mathcal{P},\textsf{T})$ is a polygonal tiling, then it is a geometric realization of $X_{(\mathbb{M},\Gamma,\mathcal{P},\textsf{T})}$ such that the edges are realized by geodesic segments of $\mathbb{M}$. 

Moreover each face of the complex $X$ can be naturally labelled by the corresponding prototile of the tiling.


\section{Dynamics on space of tilings}\label{sec:dynamics}

\subsection{Space of tilings}\label{spacetiling}
As before, we fix a subgroup $\Gamma$ of $\isom(M)$,
and we consider a finite set $\mathcal{P}$ of prototiles.
The {\bf space of tilings on} $\mathcal{P}$,
denoted by $\Sigma_{\mathcal{P}}$,
is the (possibly empty) set of all the tilings of $M$ modelled on $\mathcal{P}$.

The group $\Gamma$ naturally acts on $\Sigma_{\mathcal{P}}$ in the
following way.
Let $x=(M, \Gamma, \mathcal{P}, \textsf{T})$ be an element of 
$\Sigma_{\mathcal{P}}$, and let $g$ be an element of $\Gamma$.
We define the tiling 
$gx=(M, \Gamma, \mathcal{P}, g\textsf{T})\in \Sigma_{\mathcal{P}}$
where
$g\textsf{T}=\{gT, T\in \textsf{T}\}.$

The set $\Sigma_{\mathcal{P}}$ can be equipped with a metric topology.
The idea is that two tilings $x,y$ in $\Sigma_{\mathcal{P}}$ are close
to each other if, up to moving $y$ by an element $g\in\Gamma$ which
``does not move a lot'', $x$ and $y$ agree on a large ball of $\mathbb{M}$
centered at the origin (see for instance \cite{Bel.Ben.Gam.06,Robin.04,Sad.08,Pet.06}. 

For $r\in\R$, we denote by $B_r$ the closed ball in $(M,d)$
of radius $r$, centered at the origin $O$ of $M$
(where we use the convention $B_r=\emptyset$ if $r<0$).
For $r\geq 0$ and $g\in\isom(M)$, we define:
$$||g||_{r}=\sup_{Q\in B_r}d(Q,gQ).$$

Let $x\in \Sigma_{\mathcal{P}}$ be a tiling and 
$K\subset M$ be compact subset. 
We denote by $K[x]$ the set of all the patches of $x$ 
for which the supports contain $K$. 
In particular, $B_{1/r}[x]$ is the set of subpatches
of $x$ for which the supports contain the ball $B_{1/r}$
centered at the origin $O$.
For any $x,y\in \Sigma_{\mathcal{P}}$, we define:
$$A(x,y)=\{r>0\;| \;\exists\; p\in B_{1/r}[x], \;\exists\; q\in B_{1/r}[y],
\;\exists\; g\in \Gamma: \;gp=q, \; ||g||_{1/r}\leq r\},$$
$$d(x,y)=\inf (\{1\}\cup A(x,y)).$$
This is a well known fact that the map $d$ is a distance on $\Sigma_{\mathcal{P}}$. 
However, for completeness, we give a detailed proof of this fact in our context in the Appendix.

\subsection{Minimality}\label{subsec:min}

We note that the group $\Gamma$ acts on 
$\Sigma_\mathcal{P}$ by homeomorphisms:
$(\Sigma_{\mathcal{P}},\Gamma)$ is a dynamical system.

Let  $x=(M,\Gamma, \mathcal{P},\textsf{T}) \in \Sigma_\mathcal{P}$ be a tiling. 
The {\bf hull} of $x$, denoted by $\Omega(x)$, is the closure
in $\Sigma_\mathcal{P}$ of the orbit of $x$ under the action of $\Gamma$:
$$\Omega(x)=\overline{\Gamma x}.$$

We recall that if $M$ is a topological space, and $\Gamma$ 
a group of homeomorphisms of $M$, the dynamical system $(M,\Gamma)$ 
is {\bf minimal} if every orbit is dense in $M$: 
$$\forall x\in M, \;\overline{\Gamma x}=M.$$
The minimality of $(\Omega(x),\Gamma)$ is related to a more 
combinatorial property of the tiling $x$. 
A tiling $x$ is {\bf repetitive} if
for any protopatch $\hat{P}$ of $\mathcal{L}(x)$ there exists $r>0$ such that for 
any $g\in \Gamma$, $\hat{P}$ occurs in the ball $gB$, where $B$ is the ball of radius $r$ centered at $O$.

\begin{prop}[Gottschalk's Theorem]\label{prop:repet}
Let $x$ be a tiling, and $\Omega(x)$ be its hull.

\begin{enumerate}[(i)]
\item Let $y$ be a tiling. Then  $y\in \Omega(x)$ 
$\Longleftrightarrow$ $\mathcal{L}(y)\subseteq\mathcal{L}(x)$.
\item If the tiling $x$ is repetitive and if $y\in \Omega(x)$, then 
$\mathcal{L}(y)=\mathcal{L}(x)$ and $y$ is repetitive.
\item  If the tiling $x$ is repetitive, then the dynamical system $(\Omega(x),\Gamma)$ is minimal.
\item If the dynamical system $(\Omega(x),\Gamma)$ is minimal, and $\Omega(x)$ is a compact set, then the tiling $x$ is repetitive.
\end{enumerate}
\end{prop}

For completeness we include a proof of Proposition~\ref{prop:repet} in Appendix.

\section{Topological pre-substitutions}\label{sec:top}
In this section, we introduce the main object of this paper: 
topological pre-substitutions (Definition \ref{pre-subst-def}). 
Concretely, a topological pre-substitution can be seen as a replacement rule 
which explains how to replace a tile by a patch of tiles. 
But there is an issue to deal with: we want to be able to iterate a pre-substitution. 
For that, a notion of compatibility naturally arises: 
when two tiles are adjacent along an edge in a patch image of a given tile, 
this implies constraints on the images of the two tiles 
(if we want to be able to iterate), and so on. 
The elaboration of this leads us to recursively define the iterates  
of a topological pre-substitution and check the compatibility of 
these iterates in Section \ref{subsec-compat}.
In fact, we note in Section \ref{subsect-heredity} that the compatibility of a 
topological pre-substitution can be encoded in a finite graph.

Two examples of pre-substitutions are treated in Section~\ref{sec:subst-pav} and can be used as illustrations of the notions introduced below.

\subsection{Definition}\label{pre-subst-def}

A {\bf topological $k$-gon} ($k\geq 3)$ is a $2$-CW-complex made of one face, 
$k$ edges and $k$ vertices, which is homeomorphic to a closed disc $\mathbb{D}^2$, 
and such that the $1$-skeleton is the boundary $\mathbb{S}^1$ of the closed disc. 
A {\bf topological polygon} is a topological $k$-gon for some $k\geq 3$.

We consider a finite set $\T=\{T_1,\dots,T_d\}$ of topological polygons.
The elements of $\T$ are called {\bf tiles}, and $\T$ is called the {\bf set of tiles}.
If $T_i$ is a $n_i$-gon,
we denote by $E_i=\{e_{1,i},\dots,e_{n_i,i}\}$ the set of edges of $T_i$.

A {\bf patch} $P$ {\bf modelled} on $\T$ is a $2$-CW-complex homeomorphic to the closed disc 
$\mathbb{D}^2$ such that 
for each closed face $f$ of $P$, there exists a tile $T_i\in\T$ and a homeomorphism 
$\tau_{f}:f\rightarrow T_i$ which respects the cellular structure. 
Then $T_i=\tau_f(f)$ is called the {\bf type} of the face $f$, and denoted by $\type(f)$.

An edge $e$ of $P$ is called a {\bf boundary edge} if it is contained
in the boundary $\mathbb{S}^1$ of the disc $\mathbb{D}^2\cong P$. Such a boundary
edge is contained in exactly one closed face of $P$.
An edge $e$ of $P$ which is not a boundary edge
is called an {\bf interior edge}.
An interior edge is contained in exactly
two closed faces of $P$.

In the following definition, and for the rest of this article, 
the symbol $\sqcup$ stands for the disjoint union.

\begin{defn}
A {\bf topological pre-substitution} is a triplet 
$(\T,\sigma(\T),\sigma)$ where:
\begin{enumerate}[(i)]
	\item $\T=\{T_1,\dots,T_d\}$ is a set of tiles, 
	\item $\sigma(\T)=\{\sigma(T_1),\dots,\sigma(T_d)\}$ 
	is a set of patches modelled on $\T$,
	\item 
	$\displaystyle \sigma: \bigsqcup_{i\in\{1,\dots,d\}} T_i 
	\rightarrow \bigsqcup_{i\in\{1,\dots,d\}} \sigma(T_i)$ 
	is a homeomorphism,
	which restricts to homeomorphisms 
	$\sigma_i:T_i\rightarrow\sigma(T_i)$,
	such that the image of a vertex of $T_i$ is a vertex of 
        the boundary of $\sigma(T_i)$.
\end{enumerate}
\end{defn}

Since $\sigma$ is a homeomorphism, 
it maps the boundary of $T_i$ homeomorphically
onto the boundary of $\sigma(T_i)$.
We note that if $e$ is an edge of $T_i$ which joins the vertex
$v$ to the vertex $v'$, then $\sigma(e)$ is an edge path in the boundary 
of $\sigma(T_i)$ which joins the vertex $\sigma(v)$ to the vertex $\sigma(v')$.

For simplicity, the topological pre-substitution 
$(\T,\sigma(\T),\sigma)$ will be often denoted by $\sigma$.

\subsection{Iterating a topological pre-substitution}

Let $\mathcal{T}=\{T_1,\dots, T_d\}$ be the set
of tiles of $\sigma$, and let 
$E_i=\{e_{1,i},\dots,e_{n_i,i}\}$ be the
the set of edges of $T_i$ ($i\in\{1,\dots,d\}$).
We denote by $E$ the set of all edges of all
the tiles of $\mathcal{T}$: 
$\displaystyle{E=\cup_i E_i}$.

\subsubsection{Balanced edges}

A pair $(e,e')\in E\times E$ is {\bf balanced} if
$\sigma(e)$ and $\sigma(e')$ have the same length
(= the number of edges in the edge path).

The {\bf flip} is the involution of $E\times E$
defined by $(e,e')\mapsto (e',e)$.
The quotient of $E\times E$ obtained by identifying 
a pair and its image by the flip
is denoted by $E_2$.
We denoted by $[e,e']$ the image of a pair
$(e,e')\in E\times E$ in $E_2$.
Since the flip preserves balanced pairs,
the notion of ``being balanced" is well defined
for elements of $E_2$.
The subset of $E_2$ which consists of balanced
elements is called the {\bf set of balanced pairs}, and
denoted by $\mathcal{B}$.

Let $[e,e']\in \mathcal{B}$ a balanced pair.
In other words, $\sigma(e)$ and $\sigma(e')$
are paths of edges which have same length say $p\geq 1$:
$\sigma(e)=e_1\dots e_p$, $\sigma(e')=e'_1\dots e'_p$
(with $e_i,e'_i\in E$ for $i\in\{1,\dots,p\}$). 
Then the $[e_i,e'_i]\in E_2$
($i\in\{1,\dots,p\}$) are called the {\bf descendants}
of $[e,e']$.

Now, we consider a patch $P$ modelled on $\mathcal{T}$. 
An interior edge $e$ of $P$ defines an 
element $[\varepsilon,\varepsilon']$ of $E_2$. 
Indeed, let $f$ and $f'$ be the two faces adjacent to $e$ in $P$.
We denote by $\varepsilon=\tau_f(e)$ the edge of $\type(f)$
corresponding to $e$, and by $\varepsilon'=\tau_{f'}(e)$
the edge of $\type(f')$ corresponding to $e$.
The edge $e$ is said to be {\bf balanced} 
if $[\varepsilon,\varepsilon']$ is balanced.

\subsubsection{Compatible pre-substitution}\label{subsec-compat}

We define, by induction on $p\in\N$, 
the notion of a $p$-compatible topological pre-substitution $\sigma$.
To any $p$-compatible topological pre-substitution $\sigma$
we associate a new pre-substitution which will be denoted by
$\sigma^p$.

\begin{defn}\label{defpcomp}
\begin{enumerate}[(a)]
\item Any pre-substitution $(\T,\sigma(\T),\sigma)$ 
is {\bf $1$-compatible}. 
\item A pre-substitution $(\T,\sigma(\T),\sigma)$ is said 
to be {\bf $p$-compatible} ($p\geq 2$) if:
\begin{enumerate}[(i)]
\item $(\T,\sigma(\T),\sigma)$ is $(p-1)$-compatible 
\item for all $i\in\{1,\dots,d\}$, 
every interior edge $e$ of $\sigma^{p-1}(T_i)$
is balanced.
\end{enumerate}
\end{enumerate}
\end{defn}

We suppose now that $(\T,\sigma(\T),\sigma)$ is a 
$p$-compatible pre-substitution. 
Then we define $\sigma^p(T_i)$ ($i\in\{1,\dots,d\}$) as
the patch obtained in the following way: 

We consider the collection of patches $\sigma(\type(f))$ for each
face $f$ of $\sigma^{p-1}(T_i)$. 
Then, if $f$ and $f'$ are two faces of $\sigma^{p-1}(T_i)$ adjacent 
along some edge $e$, we glue, edge to edge, 
$\sigma(\type(f))$ and $\sigma(\type(f'))$
along $\sigma(\tau_{f}(e))$ and $\sigma(\tau_{f'}(e))$. 
This is possible since the $p$-compatibility of $\sigma$
ensures that the edge $e$ is balanced.
The resulting patch $\sigma^p(T_i)$ is defined by:

$$\sigma^p(T_i)=\left.\left(\bigsqcup_{f \text{ face of } \sigma^{p-1}(T_i)}
\sigma(\type(f)) \right) \right/ \sim$$
where $\sim$ denotes the gluing.

We define $\sigma^p(\T)$ to be the set 
$\{\sigma^p(T_1),\dots,\sigma^p(T_d)\}$.

The map $\sigma$ induces a natural map on the faces of each $\sigma^{p-1}(T_i)$ which factorizes to a map
$\sigma_{i,p}:\sigma^{p-1}(T_i) \rightarrow \sigma^{p}(T_i)$
thanks to the $p$-compatibility hypothesis:

\begin{equation*}
\begin{CD}
  \displaystyle\bigsqcup_{f \text{ face of } \sigma^{p-1}(T_i)}
\type(f)  @> \sigma >> \displaystyle\bigsqcup_{f \text{ face of } \sigma^{p-1}(T_i)}
\sigma(\type(f)) \\
@V{\sim}VV        @VV{\sim}V\\
 \sigma^{p-1}(T_i) @>>\sigma_{i,p}> \sigma^{p}(T_i)
\end{CD}
\end{equation*}

We note that $\sigma_{i,p}$ is a homeomorphism which
sends vertices to vertices.
Then we define the map
$$\sigma^{p}_i: T_i \rightarrow \sigma^{p}(T_i)$$
as the composition:
$\sigma^{p}_i= \sigma_{i,p}\circ\sigma_i^{p-1}$.
This is an homemorphism which
sends vertices to vertices.
Then $\sigma^p$ is naturally defined
such that the restriction of $\sigma^p$
on $T_i$ is $\sigma_i^p$.
We remark that $\sigma^1=\sigma$.

We have thus obtained a topological pre-substitution 
$(\T,\sigma^p(\T),{\sigma}^p)$.

\begin{rem}
We can define $\sigma^0$ to be the identity to
$\bigsqcup_i T_i$.
Thus, setting that $(\T,\T,\sigma^0)$ is $0$-compatible, 
we consistently extend Definition \ref{defpcomp} 
to all non-negative integers $p$.

\end{rem}

\begin{example}
We give an example of a pre-substitution $\sigma$ which is $1$-compatible 
but not $2$-compatible. The set of tiles consists of only one tile $T$, which is a triangle.
The vertices of $T$ are denoted $t_1,t_2,t_3$ 
(or simply $1,2,3$ on Figure \ref{fig:triangle})
and the edge between 
$t_i$ and $t_{i+1}$ is denoted $t_{i,i+1}$, (for $i\in\{1,2,3\}$ modulo 3). 
The patch $\sigma(T)$ is obtained by gluing two copies of $T$ along 
the edges $t_{12}$ for one and $t_{13}$ for the other one. 
The image $\sigma(t_i)$ of the vertex $t_i$ is indicated on the figure.
 This pre-susbtitution is not $2$-compatible since 
$|\sigma(t_{13})|=2, |\sigma(t_{12})|=1$, see Figure \ref{fig:triangle}.
\end{example}

 \begin{figure}[!ht] 
 \begin{center} 
 \begin{tikzpicture}[scale=.6]
\draw (0,0)--(4,0)--(2,3.236)--cycle;
\draw (0.7,.4) node{{\tiny1}};
\draw (3.3,.4) node{{\tiny 2}};
\draw (2,2.6) node{{\tiny 3}};
\draw (2,1) node{T};

\draw (0,0) node{$\bullet$};
\draw (4,0) node{$\bullet$};
\draw (2,3.236) node{$\bullet$};

\draw (4,1.5)--(5.5,1.5);
\draw (5.3,1.6)--(5.5,1.5);
\draw (5.3,1.4)--(5.5,1.5);

\draw (6,0)--(10,0)--(8,3.236)--cycle;
\draw (10,0)--(8,3.236)--(12,3.236)--cycle;
\draw (6.7,.4) node{{\tiny 3}};
\draw (9.3,.4) node{{\tiny 1}};
\draw (8,2.6) node{{\tiny 2}};
\draw (11.3,2.8) node{{\tiny 2}};
\draw (10,.6) node{{\tiny 1}};
\draw (8.7,2.8) node{{\tiny 3}};
\draw (8,1) node{T};
\draw (10,2) node{T}; 
\draw (6.1,-.7) node{$\sigma(1)$};
\draw (9.7,-.7) node{$\sigma(2)$};
\draw (12.3,3.7) node{$\sigma(3)$};

\draw (6,0) node{$\bullet$};
\draw (10,0) node{$\bullet$}; 
\draw (12,3.236) node{$\bullet$};
 
 \draw (12,1.5)--(13.5,1.5);
\draw (13.3,1.6)--(13.5,1.5);
\draw (13.3,1.4)--(13.5,1.5);
 
\draw (16,0)--(14,3.236)--(18,3.236)--cycle;
\draw (16,0)--(20,0)--(18,3.236)--cycle;
\draw (20.3,0)--(18.3,3.236)--(22.3,3.236)--cycle;  
\draw (22.3,3.236)--(19.3,1.618);  
  
\draw (18,2.7) node{{\tiny 1}};
\draw (16.7, .4) node{{\tiny 2}};
\draw (16,.6) node{{\tiny 3}};
\draw (17.3,2.8) node{{\tiny 1}};
\draw (19.2,.4) node{{\tiny 3}};
\draw (20.3,.6) node{{\tiny 3}};  
\draw (21.5,2.4) node{{\tiny 1}};
\draw (19.8,1.4) node{{\tiny 2}};  
\draw (19.4,2.1) node{{\tiny 3}};
\draw (20.9,2.85) node{{\tiny 1}};
\draw (18.9,2.8) node{{\tiny 2}};
 \draw (14.7,2.8) node{{\tiny 1}};
 \end{tikzpicture}

 \end{center}
 \caption{A topological pre-substitution which is not {\tiny 2}-compatible}
 \label{fig:triangle}
 \end{figure}
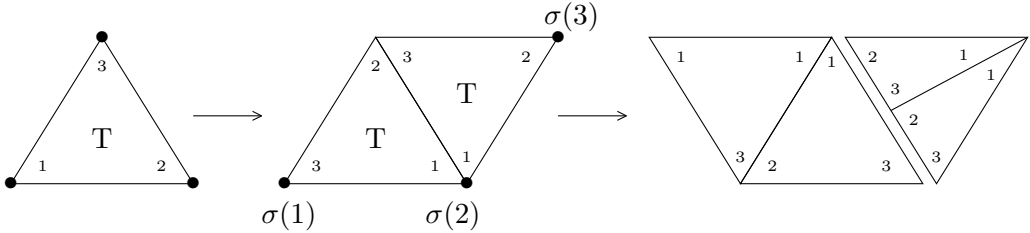

\begin{defn}
A pre-substitution $\sigma=(\T,\sigma(\T),\sigma)$ is 
{\bf compatible} if it is $p$-compatible for every 
$p\in\N$.
\end{defn}

\begin{rem} 
If the map $e\mapsto |\sigma(e)|$ is constant  on the set $E$, then $\sigma$ is compatible. Let  $\sigma=(\T,\sigma(\T),\sigma)$ be a compatible pre-substitution and let $T\in\mathcal{T}$. 
For all $p,q\in\mathbb{N}$, we will denote by $\sigma^{q}$ the natural map defined by
$\sigma^q:\sigma^p(T)\rightarrow \sigma^{p+q}(T)$.
\end{rem}

\subsubsection{Heredity graph of edges $\mathcal{E}(\sigma)$}\label{subsect-heredity}

In this subsection we give an algorithm which decides
whether a pre-substitution is compatible. Examples of heredity graphs are computed in Section~\ref{sec:subst-pav}.

Suppose that $\sigma$ is $p$-compatible.
We define $W_p$ as the set of elements 
$[\varepsilon,\varepsilon']\in E_2$ such that there
exists $i\in\{1,\dots, d\}$, as well as two 
faces $f$ and $f'$ in $\sigma^p(T_i)$ glued along 
an edge $e$, such that
$\tau_f(e)=\varepsilon$ and $\tau_{f'}(e)=\varepsilon'$. 
The topological pre-substitution $\sigma$ is 
$(p+1)$-compatible if and only if  $W_p$ is contained in the 
set of balanced pairs $\mathcal{B}$:
$W_p\subseteq \mathcal{B}$. Then:

\begin{itemize}
\item either $W_p\nsubseteq \mathcal{B}$: the
algorithm stops, telling us that the substitution is
not compatible,
\item or $W_p\subseteq \mathcal{B}$: then
we define $V_p=V_{p-1}\cup W_p$.
\end{itemize}
By convention we settle $V_0=\emptyset$.

Suppose that $\sigma$ is compatible.
The sequence $(V_p)_{p\in\N}$ is an increasing
sequence (for the inclusion) of subsets of
the finite set $E_2$. Hence
there exists some $p_0\in\N$ such that
$V_{p_0+1}=V_{p_0}$ (and thus 
$V_{p}=V_{p_0}$ for all $p\geq p_0$).
The algorithm stops at step $p_0+1$
(where $p_0$ is the smallest integer
such that $V_{p_0+1}=V_{p_0}$),
telling that $\sigma$ is compatible.

The {\bf heredity graph of edges} of $\sigma$,
denoted $\mathcal{E}(\sigma)$, is defined in the following 
way. The set of vertices of $\mathcal{E}(\sigma)$
is $V_{p_0}$. There is an oriented edge from
vertex $[e,e']$ to vertex $[\epsilon,\epsilon']$
if $[\epsilon,\epsilon']$ is a descendant of $[e,e']$.


\section{Topological substitutions and substitutive tilings}\label{sec:top-subst}
We introduce in this section the notion of ``core property" which is natural to obtain an expanding dynamical system.
\subsection{The core property}
Let $P$ be a patch modelled on $\mathcal{T}=\{T_1,\dots,T_d\}$.
The {\bf thick boundary} $\B(P)$ of $P$ 
is the closed sub-complex of $P$ consisting of 
the closed faces which contain at least one vertex 
of the boundary $\partial P$ of $P$.
The {\bf core} $\core(P)$ of $P$ is 
the closure in $P$ of the complement
of $\B(P)$: in particular, $\core(P)$ is a closed subcomplex of $P$
-- see Figure~\ref{fig:thick-core}.

\begin{figure}[!ht] 
 \begin{center} 
 \includegraphics[width=5cm]{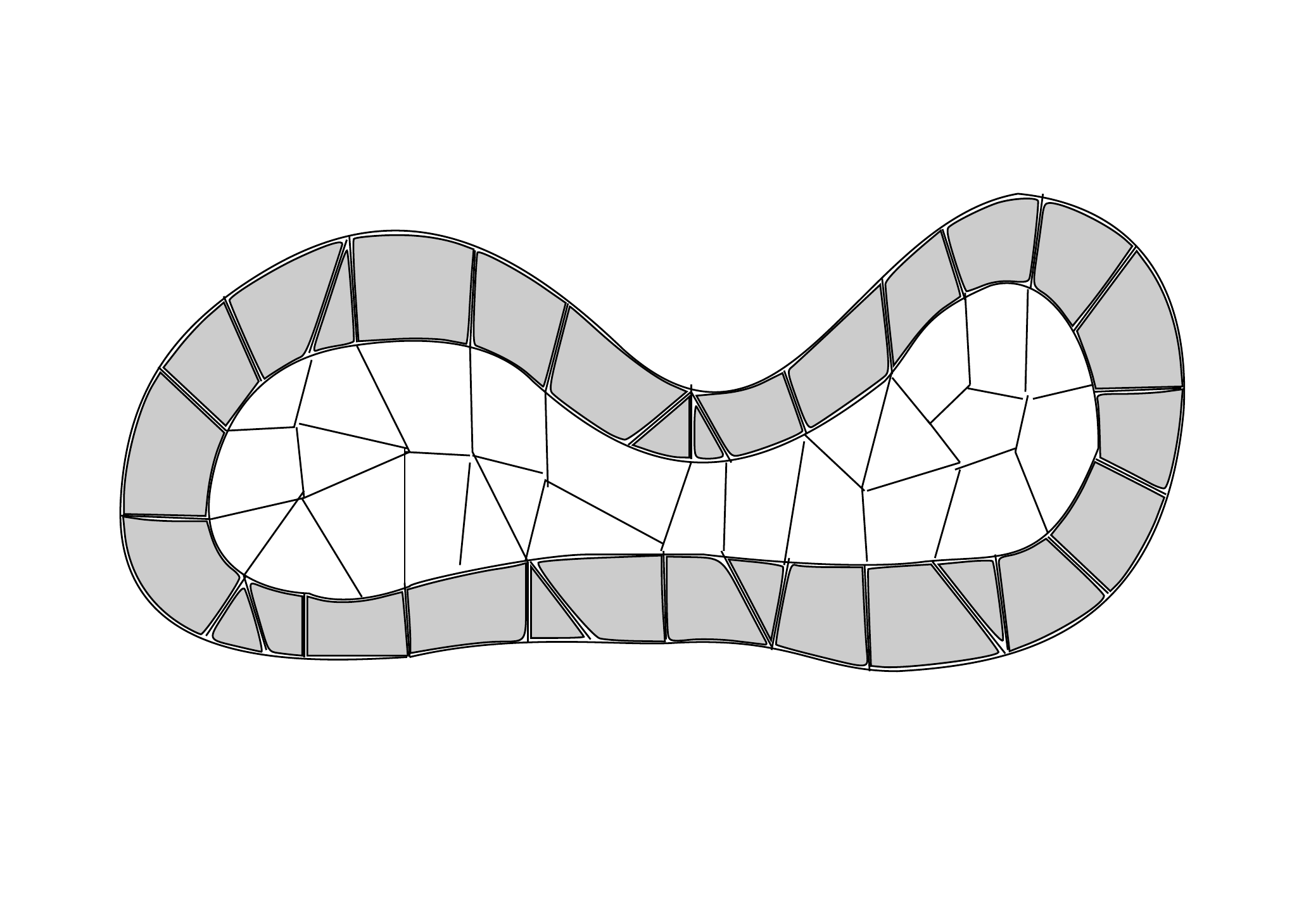}
 \caption{The thick boundary is the grey subcomplex and the core is the white subcomplex}
 \label{fig:thick-core}
 \end{center}
 \end{figure}
 
\begin{rem}\label{rem:core}
Let $(\T,\sigma(\T),\sigma)$ be a pre-substitution. 
It follows directly from the definition of $\sigma^p$
that
for any $T\in\T$, and any $p\in\N$, one has:
$$\sigma(\core(\sigma^p(T)))\subseteq \core(\sigma^{p+1}(T)),$$
$$\B(\sigma^{p+1}(T))\subseteq\sigma(\B(\sigma^p(T))).$$
\end{rem}

\begin{defn}
A pre-substitution $(\T,\sigma(\T),\sigma)$ has the {\bf core property} 
if there exist $i\in\{1\dots d\}$ and $k\in\mathbb{N}$ such that the core 
of $\sigma^k(T_i)$ 
is non-empty.
\end{defn}
For simplicity, in the following we will assume that $i=1$.

\begin{example}
The substitution defined by Figure \ref{carre} does not have the core property.
Indeed, the thick boundary of $\sigma^n(T)$ is equal to $\sigma^n(T)$ for every integer $n$. 

 \begin{figure}[ht]
\begin{center} 
 \includegraphics[width= 9cm]{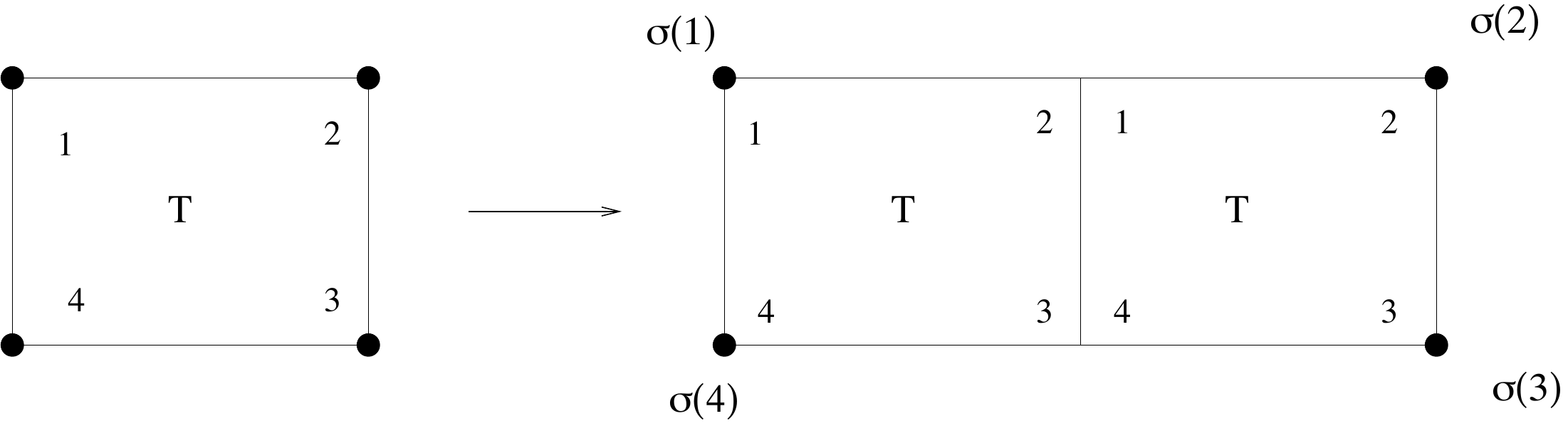}
 \end{center}

\caption{For every integer $k$, the $2$-complex $\sigma^k(T)$ has empty core}\label{carre}
 \end{figure}

\end{example}

\begin{defn}\label{def:top-sub}
A {\bf topological substitution} is a pre-substitution 
which is compatible and has the core property. 
\end{defn}

\subsection{Inflation}\label{sec:inflation}

Consider a tile $T\in\mathcal{T}$ such that the core of 
$\sigma(T)$ contains a face of type $T$.
Then, we can identify the tile  $T$ with a subcomplex of 
the core of $\sigma(T)$. 
By induction, $\sigma^k(T)$ is thus identified with a subcomplex 
of $\sigma^{k+1}(T)$ ($k\in\mathbb{N}$). 
We define $\sigma^\infty(T)$ as the increasing union:

$$\sigma^\infty(T)=\bigcup_{k=0}^\infty\sigma^k(T).$$
By construction, the complex $\sigma^\infty(T)$ 
is homeomorphic to $\mathbb{R}^2$.
We say that such a complex is obtained {\bf by inflation from $\sigma$}. 
Moreover this complex can be labelled by the types of the topological polygons.

\begin{defn}
A tiling of the plane $\mathbb{M}$
is {\bf substitutive} if the labelled complex associated to it (see Section \ref{section:complex})
can be obtained by inflation from a topological substitution. 
In this case the geometric realization of a tile $T$ is denoted $\lfloor T\rfloor$.
\end{defn}

We remark that in this case, the substitution can be realized in a map $\lfloor \sigma \rfloor$.

\begin{example}
The regular tiling of $\mathbb{E}^2$ by equilateral triangles is substitutive. Indeed the associated complex can be obtained by inflation from the topological substitution $\sigma$ defined in Figure \ref{triang-equ}. We  notice that the core of $\sigma(T)$ is empty, but the core of $\sigma^2(T)$ is nonempty.
\end{example}
 \begin{figure}[ht]
\begin{center} 
\begin{tikzpicture}[scale=.4]
\draw (0,0)--(4,0)--(2,3.236)--cycle;
\draw (0.7,.4) node{{\tiny 1}};
\draw (3.3,.4) node{{\tiny 2}};
\draw (2,2.6) node{{\tiny 3}};
\draw (2,1) node{{\tiny T}};

\draw (0,0) node{$\bullet$};
\draw (4,0) node{$\bullet$};
\draw (2,3.236) node{$\bullet$};

\draw (6,1.5)--(7.5,1.5);
\draw (7.3,1.6)--(7.5,1.5);
\draw (7.3,1.4)--(7.5,1.5);

\draw (8,0)--(16,0)--(12,6.472)--cycle;
\draw (12,0)--(14,3.236)--(10,3.236)--cycle;
\draw (8.7,.4) node{{\tiny 1}};
\draw (11.3,.4) node{{\tiny 2}};
\draw (10,2.6) node{{\tiny 3}};
\draw (10,1) node{{\tiny T}};

\draw (12.7,.4) node{{\tiny 1}};
\draw (15.3,.4) node{{\tiny 2}};
\draw (14,2.6) node{{\tiny 3}};
\draw (14,1) node{{\tiny T}};

\draw (10.7,3.6) node{{\tiny 1}};
\draw (13.3,3.6) node{{\tiny 2}};
\draw (12,5.7) node{{\tiny 3}};
\draw (12,4) node{{\tiny T}};

\draw (11.9,.7) node{{\tiny 1}};
\draw (10.8,2.6) node{{\tiny 3}};
\draw (13,2.6) node{{\tiny 2}};
\draw (12,2) node{{\tiny T}};

\draw (8,0) node{$\bullet$};
\draw (16,0) node{$\bullet$};
\draw (12,6.472) node{$\bullet$};

\draw (8,-1) node{$\sigma(1)$};
\draw (16,-1) node{$\sigma(2)$};
\draw (12,7.5)  node{$\sigma(3)$};

\draw (18,1.5)--(19.5,1.5);
\draw (19.3,1.6)--(19.5,1.5);
\draw (19.3,1.4)--(19.5,1.5);

\draw (20,0)--(36,0)--(28,12.948)--cycle;
\draw (28,0)--(32,6.472)--(24,6.472)--cycle;
\draw (24,0)--(26,3.236)--(22,3.236)--cycle;
\draw (32,0)--(34,3.236)--(30,3.236)--cycle;
\draw (28,6.472)--(26,3.236)--(30,3.236)--cycle;
\draw (28,6.472)--(26,9.708)--(30,9.708)--cycle;
\draw (20,0) node{$\bullet$};
\draw (36,0) node{$\bullet$};
\draw (28,12.948) node{$\bullet$};
\draw (20.7,.4) node{{\tiny 1}};
\draw (23.3,.4) node{{\tiny 2}};
\draw (22,2.6) node{{\tiny 3}};
\draw (24.7,.4) node{{\tiny 1}};
\draw (27.3,.4) node{{\tiny 2}};
\draw (26,2.6) node{{\tiny 3}};

\draw (28.7,.4) node{{\tiny 1}};
\draw (31.3,.4) node{{\tiny 2}};
\draw (30,2.6) node{{\tiny 3}};

\draw (32.7,.4) node{{\tiny 1}};
\draw (35.3,.4) node{{\tiny 2}};
\draw (34,2.6) node{{\tiny 3}};

\draw (22.7,3.6) node{{\tiny 1}};
\draw (25.3,3.6) node{{\tiny 2}};
\draw (24,5.7) node{{\tiny 3}};

\draw (26.7,3.6) node{{\tiny 3}};
\draw (29.3,3.6) node{{\tiny 1}};
\draw (28,5.7) node{{\tiny 2}};

\draw (30.7,3.6) node{{\tiny 1}};
\draw (33.3,3.6) node{{\tiny 2}};
\draw (32,5.7) node{{\tiny 3}};

\draw (24.7,6.8) node{{\tiny 1}};
\draw (27.3,6.8) node{{\tiny 2}};
\draw (26,8.9) node{{\tiny 3}};

\draw (28.7,6.8) node{{\tiny 1}};
\draw (31.3,6.8) node{{\tiny 2}};
\draw (30,8.9) node{{\tiny 3}};

\draw (26.7,10) node{{\tiny 1}};
\draw (29.3,10) node{{\tiny 2}};
\draw (28,12.1) node{{\tiny 3}};

\draw (24.8,2.6) node{{\tiny 2}};

\draw (27.9,.7) node{{\tiny 1}};
\draw (28.8,2.6) node{{\tiny 2}};
\draw (27,2.6) node{{\tiny 3}};

\draw (31.9,.7) node{{\tiny 1}};
\draw (32.8,2.6) node{{\tiny 2}};
\draw (31,2.6) node{{\tiny 3}};

\draw (23.9,.7) node{{\tiny 1}};

\draw (23,2.6) node{{\tiny 3}};

\draw (26,3.8) node{{\tiny 1}};
\draw (26.8,5.7) node{{\tiny 2}};
\draw (25,5.7) node{{\tiny 3}};

\draw (30,3.8) node{{\tiny 1}};
\draw (30.8,5.7) node{{\tiny 2}};
\draw (29,5.7) node{{\tiny 3}};

\draw (28,7) node{{\tiny 1}};
\draw (28.8,9) node{{\tiny 2}};
\draw (27,8.9) node{{\tiny 3}};
\end{tikzpicture}
 \end{center}
 \caption{A topological substitution which gives rise to the regular tiling 
 of $\mathbb{E}^2$ by equilateral triangles}
\label{triang-equ}
 \end{figure}
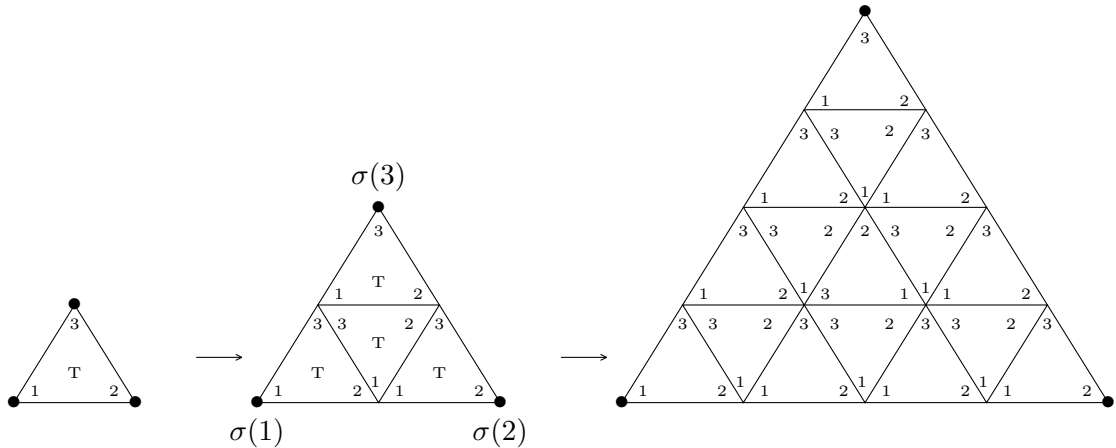

\begin{rem}\label{rem:path}
Consider a substitutive tiling 
$(\mathbb{M},\Gamma,\mathcal{P},\textsf{T})$ 
associated to the topological  substitution 
$(\mathcal{T},\sigma(\mathcal{T}),\sigma)$.
A path joining the geometric realization $\lfloor \partial\sigma^n(T) \rfloor$
of the boundary of $\sigma^n(T)$ to
the geometric realization $\lfloor \partial\sigma^{n+k}(T) \rfloor$
of the boundary of $\sigma^{n+k}(T)$ 
intersects at least $k$ tiles.  Indeed such a patch has to intersect the 
geometric realizations of the thick boundaries $B(\sigma^{k+i}(T))$ $(i=1\dots n)$, which are disjoint.
\end{rem}

\begin{rem}\label{rem:supertile}
Consider a substitutive tiling 
$(\mathbb{M},\Gamma,\mathcal{P},\textsf{T})$ 
associated to the topological  substitution 
$(\mathcal{T},\sigma(\mathcal{T}),\sigma)$. 
There exists a substitutive tiling
$(\mathbb{M},\Gamma,\mathcal{P}_k,\textsf{T}_k)$ 
associated to the topological  substitution 
$(\sigma^{k}(\mathcal{T}),\sigma(\sigma^{k}(\mathcal{T})),\sigma)$.  
Each element of $\mathcal{P}$ is the geometric realization of a 
$T_i$ $(i=1\dots d)$, and each element of $\mathcal{P}_k$ is the 
geometric realization of a patch $\sigma^{k}(T_i)$ $(i=1\dots d)$. 
An element of $\mathcal{P}_k$ is called a {\bf super-tile} of 
order $k$.
\end{rem}


\section{Primitive topological substitutions}\label{sec:prim}
\subsection{Primitive matrix}\label{subsec:pf}
An integer matrix $M\in \mathcal{M}_d(\mathbb{N})$ is {\bf primitive} 
if there exists some $k\in\mathbb{N}$ such that  all entries 
of $M^k$ are positive.
A primitive matrix $M$ satisfies the Perron-Frobenius Theorem 
(see for instance \cite[Chapter 1]{Sen.06}). 
We cite here only the part of it which is relevant in the context of the paper.   

\begin{thm}[Perron-Frobenius]\label{th:PF}
Let $M\in\mathcal{M}_d(\mathbb{N})$ be a primitive matrix.
There exists a positive real number $\lambda$, called the 
Perron Frobenius eigenvalue, such that $\lambda$ is an 
eigenvalue of $M$, and any other eigenvalue $r$ (possibly complex) 
has modulus strictly smaller than $\lambda$, {\it i.e} $|r|<\lambda$. 
Moreover there exists a left eigenvector $v$ of $M$ associated with 
$\lambda$, $v = (v_1,\dots,v_d)$, with strictly 
positive coordinates:
$$vM=\lambda v\;\;\; \text{ with }\; v_i>0 \; \text{ for all } \; 1\leq i\leq n .$$
If one asks besides that $v_1+\dots+v_d=1$, then the eigenvector
$v$ is unique.
\end{thm}

Given a primitive matrix $M\in \mathcal{M}_d(\mathbb{N})$,
for $x=(x_1,\dots,x_d)\in\R^d$ we define: 
$$||x||_{PF}=\displaystyle\sum_{i=1}^dv_ix_i=v\;{}^t\!x$$
where $v$ is the unique eigenvector given by Theorem \ref{th:PF}.
The theorem allows us to compute the norm
of $M^nx$.
\begin{equation}\label{eq:M^nx}
||M^nx||_{PF}=vM^n\;{}^t\!x=\lambda^n(v\;{}^t\!x)=\lambda^n||x||_{PF}.
\end{equation}

\subsection{Primitive substitution}
The {\bf transition matrix} 
$M_\sigma\in\mathcal{M}_{d}(\mathbb{N})$ associated
to the topological substitution $\sigma$ is the matrix whose entry $m_{i,j}$ 
is the number of faces of type $T_i$ in the patch $\sigma(T_j)$.
We note that if $M_\sigma$ is primitive then there exists an integer $k$ such that for any tile 
$T\in\T$, 
the patch $\sigma^k(T)$ contains a face of each possible type in $\T$.

\begin{lem}\label{prim}
Let $(\T,\sigma(\T),\sigma)$ be a topological substitution with a primitive transition matrix. There exists an 
integer $k$ such that for all $T,T'\in\mathcal{T}$, the core of $\sigma^k(T)$ 
contains a face of type $T'$. 
\end{lem}
\begin{proof}
There exists some $k_0$ such that $\core(\sigma^{k_0}(T))\neq\emptyset$: there is a tile of type $T''$ in $\core(\sigma^{k_0}(T))$.
By primitivity, there exists some $k_1$ such that $\sigma^{k_1}(T'')$ contains 
a tile of type $T'$.
Since the core of $\sigma^{k_0+k_1}(T)$ contains the image under 
$\sigma^{k_1}$ of the core of $\sigma^{k_0}(T)$, we obtain that 
the core of $\sigma^{k_0+k_1}(T)$ contains a face of type $T'$. 
\end{proof}

\begin{defn}
Let $P$ be a patch and $P'$ be a subpatch. The patch $P$ is {\bf separable by} $P'$ if the complement of the closure of $P'$ in $P$,
$P\smallsetminus\overline{P'}$, is not connected. 
The patch $P$ is {\bf separated} if there exists a tile $T$ in $P$ such that 
$P$ is separable by $T$. 
\end{defn}
For instance, for the substitution $\sigma$ defined in Figure \ref{carre}, $\sigma(T)$ is not separated, but $\sigma^2(T)$ is separated.

\begin{defn}\label{def:pureprim}
A topological substitution $(\T,\sigma(\T),\sigma)$ is {\bf pure primitive} if:
\begin{enumerate}[(i)]
\item for all $T\in\mathcal{T}$, the patches $\sigma(T), \sigma^2(T)$ are not separated,
\item for all $T, T'\in\mathcal{T}$, the core of $\sigma(T)$ contains a face of type $T'$.
\end{enumerate}
A topological substitution $(\T,\sigma(\T),\sigma)$ is {\bf primitive} if there exists an integer $k$ such that $\sigma^k$  is pure primitive.
\end{defn}

Property (ii) of a pure primitive substitution $\sigma$ ensures that  
the complex $\sigma^{\infty}(T)$ exists for all $T\in\mathcal{T}$.

\begin{rem}
In the case of a ``usual'' substitutive tiling of the Euclidean plane, the condition that the substituted patches 
are not separated is insured by the linear expansion underlying the substitution. 
In our topological setting, in order to mimic the ``usual'' behavior,
we need to demand Property (i) as an {\em ad hoc} condition.
\end{rem}

\begin{lem}\label{lem:separable}
Let  $(\T,\sigma(\T),\sigma)$ be a pure primitive topological substitution, 
then for every integer $i\in\{1\dots d\}$ and for every integer $n\in\mathbb{N}$, 
the patch $\sigma^n(T_i)$ is not separated.
\end{lem}
\begin{proof}
We make the proof by induction on $n$. By hypothesis, it is true for $n=1, 2$. 
Consider a patch $\sigma^n(T_i)$, and assume that it is separated by a 
(closed) tile $T$.
Let $\Sigma_1, \Sigma_2,\dots, \Sigma_p$, $p\geq 2$, be the connected components of  $\sigma^n(T_i)\smallsetminus T$.

Let $T'$ be the tile in $\sigma^{n-1}(T_i)$ such that $T$ lies in $\sigma(T')$,
and let $T''$ be the tile in $\sigma^{n-2}(T_i)$ such that $T'$ lies in $\sigma(T'')$.

Note that $\sigma(T')$ has nonempty intersection with at most one of
the $\Sigma_i$'$s$, say $\Sigma_1$:
otherwise $\sigma(T')$ would be separated by $T$, which contradicts pure 
primitivity. 
Then  $\sigma(T')$ must have nonempty intersection with  $\Sigma_1$
(otherwise, $\sigma(T')$ would have empty core).
By the induction hypothesis, $\sigma^{n-1}(T_i)$ is not separated by $T'$.
Thus, since $\sigma:\sigma^{n-1}(T_i)\rightarrow\sigma^{n}(T_i)$ is a homeomorphism, 
$\sigma^{n}(T_i)$ is not separated by $\sigma(T')$.
This implies that $\sigma(T')$ is the union of $T$ and $\Sigma_1$.

The patch $\sigma^2(T'')$ is not separated by $T$.  
Arguing as previously,
$\sigma^2(T'')$ has nonempty intersection with at most one of the $\Sigma_i$'$s$. 
Moreover $\sigma^2(T'')$ contains $\sigma(T')=T\cup\Sigma_1$.
Thus $\sigma^2(T'')=\sigma(T')$, and so  $\sigma(T'')=T'$ has empty
core: a contradiction. 
\end{proof}

The {\bf norm} of a patch $P$ modelled on $\T$ is defined as: 
$$||P||=\displaystyle\sum_{i=1}^dv_i|P|_i$$ 
where $|P|_i$ denotes the number of faces of type $T_i$ 
in the patch $P$.

Using the above equality (\ref{eq:M^nx}), we obtain directly

\begin{lem}\label{lem:perron} 
Let $(\T,\sigma(\T),\sigma)$ be a primitive topological substitution.
For any tile $T\in\T$ and for any integer $n\in\mathbb{N}$,
$||\sigma^n(T)||=\lambda^n ||T||.$ 
\end{lem}

\subsection{Minimality and primitive substitutive tilings}

In this subsection we consider a substitutive tiling of $\mathbb{M}$ and its associated CW complex. We now prove that a primitive substitutive tiling generates a minimal dynamical system. We use the definitions of Section \ref{spacetiling}.

\begin{prop}\label{prop:prim->min}
Let $(\mathcal{T},\sigma(\mathcal{T}),\sigma)$ be a primitive topological 
substitution and let $x$ be a geometrical 
realization of $\sigma^\infty(T)$ (where $T$ is a tile of $\mathcal{T}$). 
Then $(\Omega(x),\Gamma)$ is minimal, where $\Omega(x)$ is the hull of the tiling $x$.
\end{prop}
\begin{proof}
According to Theorem \ref{prop:repet}, it is sufficient to prove that $x$ is repetitive.

Up to replacing $\sigma$ by a power, we can suppose that $\sigma$ is pure primitive. 
In particular, for all $i,j\in\{1\dots d\},$ the tile $ T_i$ occurs in the support 
of the realization of the patch $ \sigma(T_j)$. 
Then for all $i,j\in\{1\dots d\}$, and for all $k\in\mathbb{N}$, 
$\lfloor \sigma^{k}(T_i)\rfloor$ occurs in $\lfloor \sigma^{k+1}(T_j)\rfloor$. 
Let $r=r(k)$ be $\displaystyle\max_{1\leq i\leq d}\diam\lfloor \sigma^{k+1}(T_i) \rfloor $, 
and $B(r)$ a ball of radius at least $r$ in $\mathbb{M}$.
By Remark \ref{rem:supertile}, $\mathbb{M}$ is naturally tiled by super-tiles of order $k+1$. 
Let $S$ be such a super-tile which contains the center of $B(r)$. 
Then $S$ is contained in $B(r)$, by definition of $r$. 
Thus there exists $j\in\{1\dots d\}$ such that $\lfloor \sigma^{k+1}(T_j)\rfloor$ occurs in $B(r)$.  
Therefore for every ball $B(r)$ of radius $r$ in $\mathbb{M}$, 
and for all $i\in\{1\dots d\}$ $\lfloor \sigma^k(T_i)\rfloor$ occurs in $B(r)$. 
This proves that $x$ is a repetitive tiling, since any patch of $x$ occurs in 
$\lfloor \sigma^K(T)\rfloor$ for some $K\in\mathbb{N}, T\in \{T_1\dots T_d\}.$
\end{proof}

This result is similar to the situation in symbolic dynamics, where the fixed point 
of a substitution is minimal if and only if the substitution is a primitive one, 
see \cite{Pyth.02}, or \cite{Kurk.03}. For instance in dimension 2, see 
\cite{Solom.97,Solom.99,Robin.04}.

\section{Primitive substitutive tilings of the hyperbolic plane}\label{sec:hyp}
In this section we give the proof of Theorem \ref{thm:prim}: 

\begin{thm}\label{thm:prim}
There does not exist a primitive substitutive tiling of the hyperbolic
plane $\Hyp^2$.
\end{thm}

First of all, we recall some facts about the isoperimetric inequality in $\Hyp^2$.

\subsection{Isoperimetric inequality in $\Hyp^2$}
Let $D$ be a domain of $\Hyp^2$. We denote by $\mathcal{A}(D)$ the area of the domain $D$.
Let $\mathcal{C}$ be a piecewise $\mathcal{C}^1$ curve in $\Hyp^2$. 
We denote by $L(\mathcal{C})$ the length of $\mathcal{C}$. 
If moreover $\mathcal{C}$ is a simple closed curve, 
$\mathcal{A}(\mathcal{C})$ will denote the area of the bounded 
component of the complement of $\mathcal{C}$ in $\Hyp^2$.

\begin{prop}[Isoperimetric inequality]\label{lem:gro}
Let  $\mathcal{C}$ be a piecewise $\mathcal{C}^1$ simple curve in $\Hyp^2$. Then:
$$\mathcal{A}(\mathcal{C})\leq  L(\mathcal{C}). $$
\end{prop}
For references see for instance \cite{Gro.86,Brid.Haef.99} and the references therein.

\subsection{Proof of Theorem \ref{thm:prim}}

\begin{lem}\label{lem:bord-tout}
Let $\sigma$ be a primitive topological substitution. 
There exists $\alpha, 0<\alpha<1$, such that for every tile $T$ 
such that the core of $\sigma(T)$ contains a face of type $T$:
\begin{equation}\label{eq:B-sigma}
||\B(\sigma^{k}(T))||\leq (1-\alpha)^k||\sigma^k(T)||.
\end{equation}
\end{lem}

\begin{proof}
Let $\sigma$ be a primitive topological substitution,
and let $T$ be a tile such that the core of $\sigma(T)$ contains
a face of type $T$.
We remark that for any $k>1$
$$\B(\sigma^{k}(T))\subset \displaystyle
\bigcup_{f\; \text{face of}\; \B(\sigma^{k-1}(T))}\B(\sigma(f))$$
(see Figure \ref{fig:core}).
Thus 
\begin{equation}\label{eq:boundary}
||\B(\sigma^{k}(T))||\leq \displaystyle
\sum_{f\; \text{face of}\;  \B(\sigma^{k-1}(T))}||\B(\sigma(f))||.
\end{equation}
We define $$\alpha=\min_{i\in\{1,\dots,d\}} 
\frac{||\core(\sigma(T_i))||}{||\sigma(T_i)||}.$$
We notice that $0<\alpha<1$ (where the first inequality comes from
the primitivity of $\sigma$). 
Then we obtain for any $f$: 
$$||\B(\sigma(f))||\leq (1-\alpha)||\sigma(f)||=(1-\alpha)\lambda||f||,$$
where the last equality follows from Lemma \ref{lem:perron}.
Combining with (\ref{eq:boundary}), we obtain:
$$||\B(\sigma^{k}(T))||\leq (1-\alpha)\lambda||\B(\sigma^{k-1}(T))||.$$
We conclude, by induction, that:
$$||\B(\sigma^{k}(T))||\leq ((1-\alpha)\lambda)^k||T||.$$
Using again Lemma \ref{lem:perron}, we deduce that:

$$||\B(\sigma^{k}(T))||\leq (1-\alpha)^k||\sigma^k(T)||.$$

\end{proof}

\begin{lem}\label{lem:curv-int}
Let  $x=(\mathbb{M},\Gamma,\mathcal{P},\textsf{T})$ be a tiling. 
For every $r>0$, there exists $C=C(r)>0$, 
such that any ball of radius $r$ in $\mathbb{M}$ intersects at most $C$ tiles of $\textsf{T}$. 
\end{lem}

\begin{proof}
The set of prototiles $\mathcal{P}$ is finite. 
We choose a representative of each prototile: $T_1,\dots, T_d$. 
We put $D=\displaystyle\max_{1\leq i\leq d}{\diam{(T_i)}}$. 
Since the tiles $T_i$ have nonempty interior, 
there exists $r_0>0$ such that each tile $T_i$ contains a ball of radius $r_0$. 
We denote the area of the ball of radius $r$ in $\mathbb{M}$ by $\mathcal{A}(r)$.

We define: 
$$C(r)=\frac{\mathcal{A}(r+D)}{\mathcal{A}(r_0)}.$$

Let $B$ be a ball of radius $r$, and $N(B)$ the number of tiles of $\textsf{T}$ intersected by $B$. Then the ball of radius $r+D$ with the same center as $B$ contains $N(B)$ disjoints balls of radius $r_0$. Thus $N(B)\leq C(r).$  
\end{proof}

\begin{prop}\label{prop:long-courb}
Let $x=(\M,\Gamma, \mathcal{P},\textsf{T})$ be a tiling of $\mathbb{M}$ obtained as a geometric realization of a topological substitution $(\mathcal{T},\sigma(\mathcal{T}),\sigma): x=\lfloor \sigma^\infty(T)\rfloor$. There exist a constant $A>0$ and an integer $n_0$  such that for every integer $n> n_0$, there exists a simple closed curve $\gamma_n$ in $\M$
such that:
\begin{itemize}
\item 
$\lfloor\sigma^{n-n_0}(T)\rfloor$ is contained in the bounded component of
$\M\smallsetminus \gamma_n$, 
\item the length $L(\gamma_n)$ satisfies: $L(\gamma_n)\leq A||\B(\sigma^n(T))||.$
\end{itemize}
\end{prop}
\begin{proof}
First we number the faces of $\B(\sigma^n(T))$: $K_0,\dots,K_p$, in such a way that $K_i, K_{i+1}$ share a common edge (with the convention that $p+1=0$).
This is possible thanks to Lemma \ref{lem:separable}.

For each prototile of $\mathcal{P}$, we fix a representative and mark a point in its interior.
Using the action of $\Gamma$, each tile $T$ of $\textsf{T}$ has now a marked point $m_T$. 
Let $s_i$ be the geodesic arc joining the marked point of the geometric realization of $K_i$ 
to the marked point of the geometric realization of $K_{i+1}$: 
$s_i=[m_{\lfloor K_i\rfloor}, m_{\lfloor K_{i+1}\rfloor}]$. 
The length $L(s_i)$ of $s_i$ is bounded by $2D$, 
where $D=\displaystyle\max_{1\leq i\leq d}{\diam{(T_i)}}$. 
By Lemma \ref{lem:curv-int}, $s_i$ intersects at most $C=C(2D)$ tiles of $\textsf{T}$. 
Thus $s_i$ does not intersect the geometric realization of $\sigma^{n-n_0}(T)$ 
with $n_0$ the integer part of $(C+1)/2$, see Remark \ref{rem:path}. 
Hence the closed curve $\eta_n$ obtained by concatening the segments 
$s_0,\dots, s_p$ does not intersect the geometric realization of $\sigma^{n-n_0}(T)$. 
The curve $\eta_n$ is not necessarily simple, but it contains, 
as a subset, a simple closed curve $\gamma_n$, 
the bounded component of the complementary in $\M$ of which contains $\lfloor \sigma^{n-n_0}(T)\rfloor $.
Since $\eta_n$ is piecewise $\mathcal{C}^1$, $\gamma_n$ 
is also piecewise $\mathcal{C}^1$. 
Of course the length of $\gamma_n$ is bounded by the length of $\eta_n$.

Let $N_n$ be the number of faces in $\B(\sigma^n(T))$. 
Then $L(\gamma_n)\leq 2DN_n$.
Since $$N_n\leq ||\B(\sigma^n(T)||/\displaystyle\min_{1\leq i\leq d}{v_i},$$
using the notation of Section \ref{subsec:pf},
the Proposition is proved with
$A=2D/\displaystyle\min_{1\leq i\leq d}{v_i}$.

\end{proof}

\begin{proof}[Proof of Theorem \ref{thm:prim}]
We consider such a primitive substitutive tiling of $\mathbb{H}^2$, 
and we use the notation of Proposition \ref{prop:long-courb}. Then:
$$
\begin{array}{rcll}
  L(\gamma_n)&\leq& A||\B(\sigma^n(T))||& \text{ by Proposition \ref{prop:long-courb} }\\
  & \leq & A(1-\alpha)^n||\sigma^n(T)||& \text{ by Lemma \ref{lem:bord-tout} }\\
  & \leq & A(1-\alpha)^n\lambda^n||T||& \text{ by Lemma \ref{lem:perron}, }
\end{array}
$$
with $0<\alpha<1$.
Besides:
$$
\begin{array}{rcll}
\mathcal{A}(\gamma_n)&\geq& \mathcal{A}(\lfloor \sigma^{n-n_0}(T)\rfloor) & \text{ by the proof of Proposition \ref{prop:long-courb}}\\
  & \geq & \displaystyle\frac{\displaystyle\min_{1\leq i\leq d}\mathcal{A}(T_i)}{\displaystyle\max_{1\leq i\leq d}v_i}||\sigma^{n-n_0}(T)|| & \text{ using notations of Section \ref{subsec:pf}  }\\
  & \geq & \frac{\displaystyle\min_{1\leq i\leq d}\mathcal{A}(T_i)}{\displaystyle\max_{1\leq i\leq d}v_i}\lambda^{n-n_0}||T||. & \text{ by Lemma \ref{lem:perron}.  }
\end{array}
$$
We derive $$L(\gamma_n)\leq C(1-\alpha)^n\mathcal{A}(\gamma_n)$$
with $C=A\lambda^{n_0}\frac{\displaystyle\max_{1\leq i\leq d}v_i}{\displaystyle\min_{1\leq i\leq d}\mathcal{A}(T_i)}$.
This is a contradiction to the isoperimetric inequality of Lemma~\ref{lem:gro}.
\end{proof}

\begin{figure}[ht]
\begin{center} 
\includegraphics[width= 11cm]{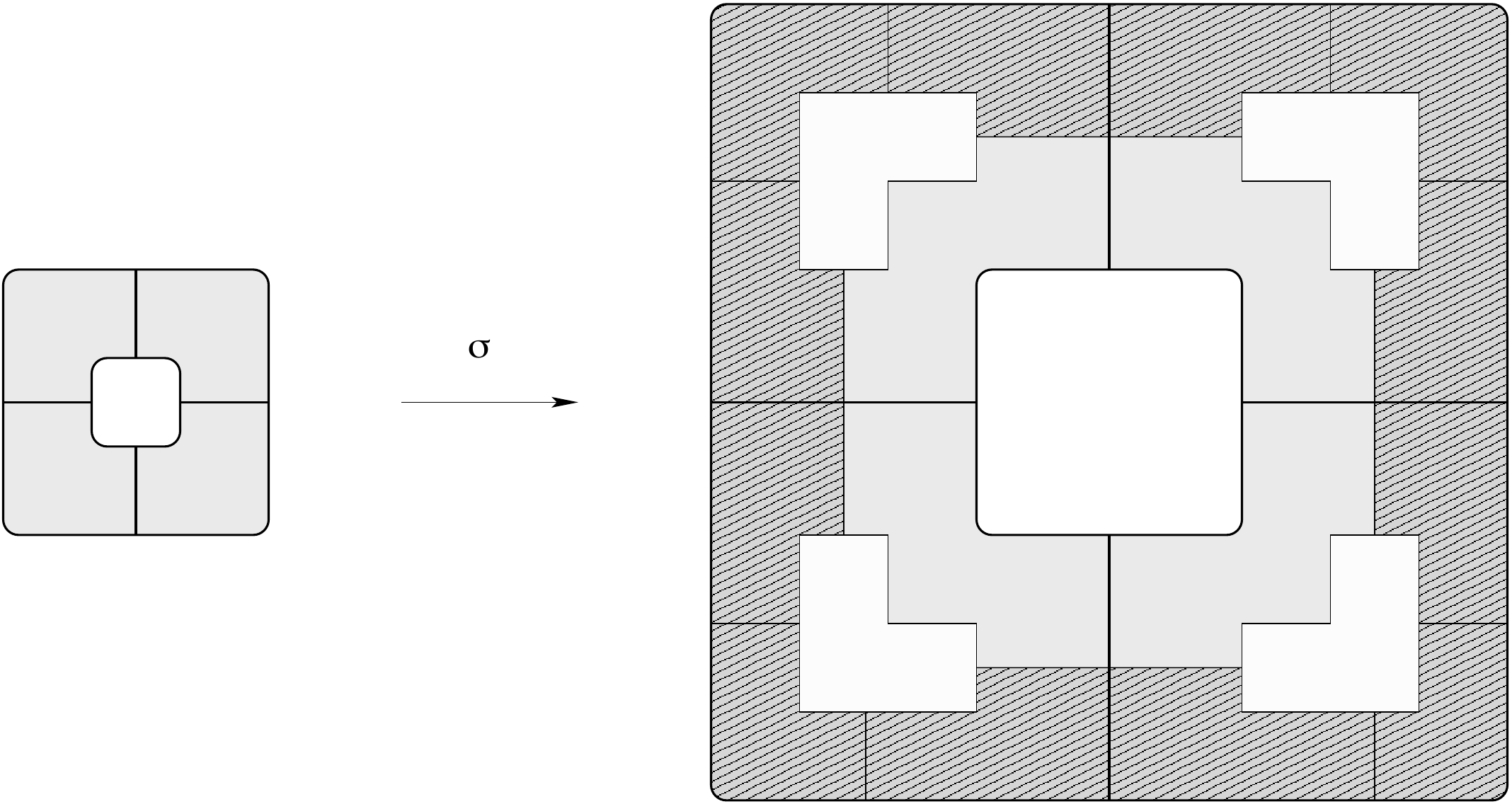}
\end{center}
\caption{The thick boundary of $\sigma^{k}(T)$ (hatched on the picture)
is contained in the union of the thick boundaries of the images of the
tiles in the thick boundary of $\sigma^{k-1}(T)$ (in grey on the picture).
}\label{fig:core}
\end{figure}

\section{Bounded valence}\label{sec:bound}

Let $(\M,\Gamma, \mathcal{P},\textsf{T})$ be a polygonal
tiling of the plane $\M$, and
let $X$ be the complex associated to this tiling.
We suppose that the set of prototiles $\mathcal{P}$
is finite.
Then the valence of each vertex of $X$ is bounded 
by $2\pi/\theta$, where $\theta>0$ is the minimum
of the angles of the (polygonal) prototiles.

Let $(\mathcal{T},\sigma(\mathcal{T}),\sigma)$ be a 
topological substitution, and let 
$X=\sigma^\infty(T)$ a complex obtained by inflation
from $\sigma$.
A necessary condition for $X$ 
to be geometrizable is that the valence of the
vertices of $X$ is bounded. In this section, we investigate this topic. 
For that, we introduce two graphs $\mathcal{V}(\sigma)$
and $\mathcal{C}(\sigma)$ associated to $\sigma$. 
We need the notions introduced in this section to study in Section \ref{sec:subst-pav} 
the substitutions $\alpha$ and $\beta$. We will give explicitly the associated 
graphs $\mathcal{C}(\alpha)$ and $\mathcal{C}(\beta)$.

\subsection{Heredity graph of vertices  $\mathcal{V}(\sigma)$}

Let $(\mathcal{T},\sigma(\mathcal{T}),\sigma)$
be a topological substitution.
We denote by $V$ the set of vertices of the tiles
of $\mathcal{T}$:
$$V=\bigcup_{T\in \mathcal{T}}\{v:\; v \textrm{ vertex of } T\}.$$
The {\bf heredity graph of vertices} is an oriented
graph denoted by 
$\mathcal{V}(\sigma)$. 
The set of vertices of $\mathcal{V}(\sigma)$
is the set $V$. Let $v\in V$ be a vertex of a tile $T\in\mathcal{T}$ and $v'\in V$ be a vertex of a tile $T'$.
There is an oriented edge from $v$ to $v'$ in  $\mathcal{V}(\sigma)$
if $\sigma(v)$ is a vertex of type $v'$ of a tile of type $T'$.

\subsection{Control of the valence}
A vertex $v\in V$ is a {\bf divided vertex} if there are
at least two oriented edges in $\mathcal{V}(\sigma)$
coming out of $v$. 
We denote by $V_D$ the subset of $V$ which consists
of all divided vertices.

Let $M$ be the maximum of the valence of the vertices
in the image of a tile:
$$M=\max\{\val(v)\mid 
v \textrm{ a vertex of } \sigma(T),\; T\in\mathcal{T}\}.$$
We examine what happens to the valence of 
vertices when passing from
$\sigma^k(T)$ to $\sigma^{k+1}(T)$.
Let $v$ be a vertex of $\sigma^{k+1}(T)$.

\underline{First case}: $v$ is not the image of a 
vertex of $\sigma^k(T)$.
Then, 
\begin{itemize}
\item either $v$ is a vertex in the interior of the
image of a tile of $\sigma^k(T)$,
and in this case, $\val(v)\leq M$;
\item or $v$ is a vertex in the boundary of the image
of a tile of $\sigma^k(T)$.
Since $v$ is not the image of a vertex of 
$\sigma^k(T)$, $v$ is in the image of the interior
of an edge. This edge is common to at most
two tiles of $\sigma^k(T)$.
Hence, $\val(v)\leq 2M$.
\end{itemize}

\underline{Second case}: $v$ is the image of a 
vertex $v_0$ of $\sigma^k(T)$; $v=\sigma(v_0)$.
Then $\val(v)\geq\val(v_0)$. Moreover, the inequality
is strict if and only if 
$v_0$ is a vertex of a face of $\sigma^k(T)$,
for which $v_0$ is a vertex with type in $V_D$. 

\begin{prop}\label{prop:Hvertex}
The following properties are equivalent.
\begin{enumerate}[(i)]
\item \label{BV} The complex $\sigma^\infty(T)$ 
has bounded valence.
\item \label{inf-SD} Every infinite oriented path in 
$\mathcal{V}(\sigma)$ crosses 
only finitely many vertices of $V_D$.
\item \label{cyc-SD} The oriented cycles of 
$\mathcal{V}(\sigma)$ 
do not cross any vertex of $V_D$.
\end{enumerate}
\end{prop}

\begin{proof}
The equivalence of conditions (\ref{BV}) and (\ref{inf-SD})
follows from the discussion preceding the Proposition.
The equivalence of conditions (\ref{inf-SD}) and (\ref{cyc-SD})
is an elementary fact about finite oriented graphs.
\end{proof}

\begin{rem}
Proposition \ref{prop:Hvertex} gives
an algorithm to check whether a given topological substitution
generates a complex with bounded valence.
It is sufficient to build the graph $\mathcal{V}(\sigma)$,
and then to check that the (finite number of) 
elementary oriented cycles do not cross $V_D$.
\end{rem}

\begin{figure}[ht]
\begin{center} 
\includegraphics[width= 13cm]{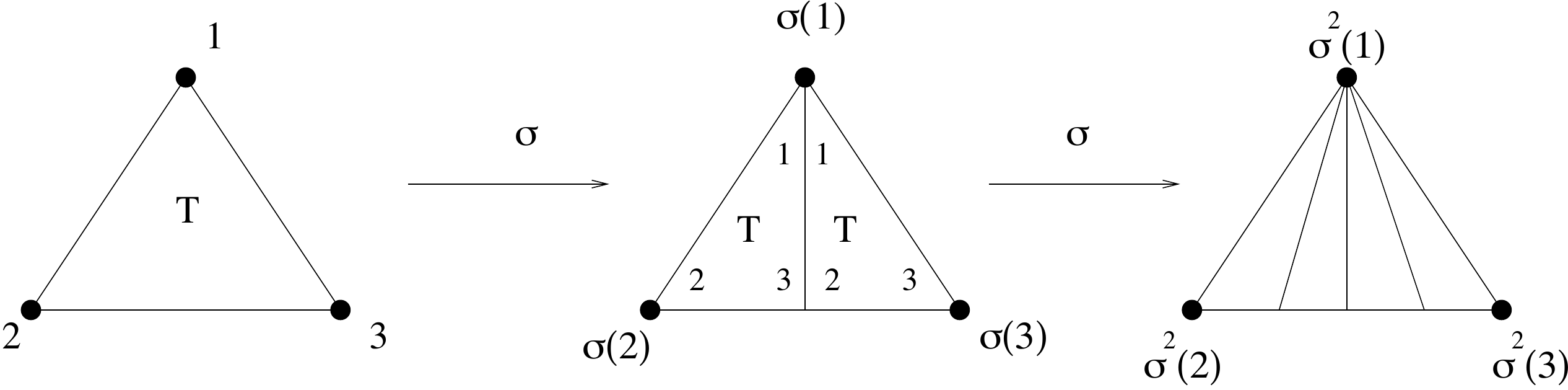}
\end{center}
\caption{The vertex $1$ of the tile $T$ of the above substitution is a divided vertex. The valence of $\sigma^k(1)$ in $\sigma^k(T)$ is $2^k+1$}\label{fig:val-expl}
\end{figure}

\subsection{Configuration graph of vertices $\mathcal{C}(\sigma)$}

Let $(\mathcal{T},\sigma(\mathcal{T}),\sigma)$ be a topological susbtitution  
such that $\sigma^{\infty}(T)$ has bounded valence.
One can be interested in understanding more precisely the evolution
of the neighborhood of the vertices when iterating $\sigma$.
For that, we introduce a new graph, which is called
the {\bf configuration graph of vertices} and denoted by 
$\mathcal{C}(\sigma)$.

We consider the equivalence relation on $k$-tuples ($k\in\N$) generated by:
$$(x_1,\dots,x_k)\sim(x_k,x_1,\dots,x_{k-1}),$$
$$(x_1,\dots,x_k)\sim(x_k,x_{k-1},\dots,x_1).$$
Let $[x_1,\dots,x_k]$ denote the equivalence class of $(x_1,\dots,x_k)$.

Let $W$ be the set of equivalence classes of $k$-tuples with $2\leq k\leq K$,
where $$K=\displaystyle\max_{\text{v vertex of}\; \sigma^{\infty}(T)} val(v).$$ 
A vertex $v$ in the interior of a patch $\sigma^n(T), (n\geq 1,T\in\mathcal{T})$ 
defines an element $[x_1,\dots,x_k]\in W$ 
(where $k$ is the valence of $v$). 
Indeed, the faces adjacent to $v$ are cyclically ordered, 
and $x_i$ is the type of the vertex of the $i$th face which is glued on $v$. 
We define the oriented graph $\mathcal{C}(\sigma)$ as follows.
The set of vertices of $\mathcal{C}(\sigma)$ is the subset $W_0$ of $W$ defined 
by the vertices in the interior of some $\sigma^n(T)$ for $n\geq 1, T\in\mathcal{T}$. 

For any $s\in W_0$, we choose some $T\in \mathcal{T}, n\geq 1$ and $v$ a vertex in the interior of $\sigma^n(T)$ which defines $s$. 
Let $s'$ the element of $W_0$ defined by $\sigma(v)$.
There is an oriented edge in $\mathcal{C}(\sigma)$ from $s$ to $s'$. 
Remark that this construction does not depend of the choice of $T$ and $n$.


\section{Substitutive tilings of $\mathbb{H}^2$}\label{sec:subst-pav}

In this section we study two examples of topological substitutions. The first one does not lead to a tiling of the plane $\Hyp^2$, but it gives rise to a tiling of an unbounded convex subset of $\Hyp^2$.The second one gives rise to a substitutive tiling of $\Hyp^2$.
\subsection{A first attempt: the substitution $\alpha$}

We consider the pre-substitution $\alpha$ defined on Figure~\ref{fig:hpara}.

\begin{figure}[ht!]
\begin{center}
\includegraphics[width= 8cm]{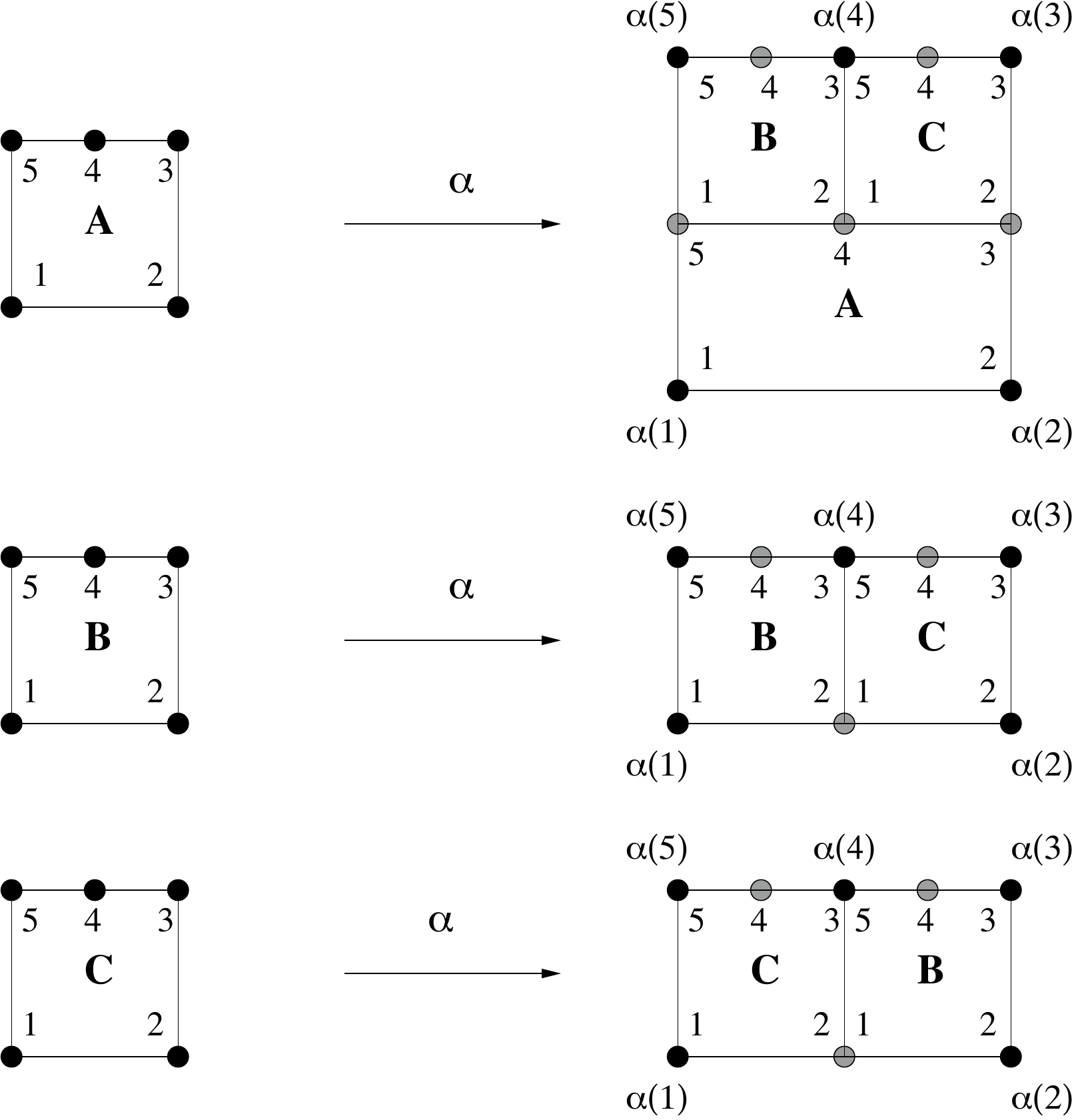}
\end{center}
\caption{Definition of the substitution $\alpha$}
\label{fig:hpara}
\end{figure}

The compatibility of $\alpha$ is checked by computing the heredity
graph of edges of $\alpha$: see Figure~\ref{fig:hpara-edge}.
\begin{figure}[ht!]
\begin{center}
$\xymatrix{
& a_{34}c_{21}\ar@(ur,ul)[d]\ar[dl] & &\\
c_{34}b_{21} \ar@(ul,ur)[d]\ar@/^-0.5pc/[dr]
& c_{45}c_{21} \ar[l] \ar@(ur,u)[] 
& b_{23}c_{15}\ar@(ul,ur)[d] & b_{15}c_{23}\ar@(ul,ur)[d] \\
b_{45}b_{21} \ar[r] \ar@(ul,u)[] 
& b_{34}c_{21}\ar[u] \ar@/^-0.5pc/[ul] & 
c_{23}c_{15} \ar@(ul,ur)[u] & b_{15}c_{23} \ar@(ul,ur)[u] \\
a_{45}b_{21}\ar[u]\ar[ur] & &&
}$
\caption{The heredity graph of edges $\mathcal{E}(\alpha)$}\label{fig:hpara-edge}
\end{center}
\end{figure}
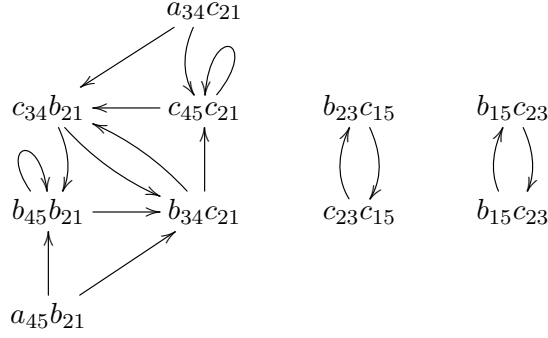

\begin{figure}[b]
\begin{center}
$\xymatrix{
a_4c_1b_2 \ar[d] & c_4b_1c_2 \ar[d] & b_4c_1b_2 \ar[d] \ar[d] \\
b_3c_5c_1c_2 \ar@(ul,ur)[d] & c_3b_5b_1b_2 \ar@(ul,ur)[d] & b_3c_5c_1c_2 \ar@(ul,ur)[d]  \\
c_3c_5c_1b_2 \ar@(ul,ur)[u] & b_3b_5b_1c_2 \ar@(ul,ur)[u] & 
c_3c_5c_1b_2 \ar@(ul,ur)[u] 
}$
\caption{The configuration graph of vertices $\mathcal{C}(\alpha)$}\label{fig:hpara-val}
\end{center}
\end{figure}
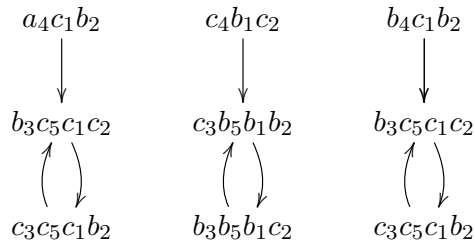

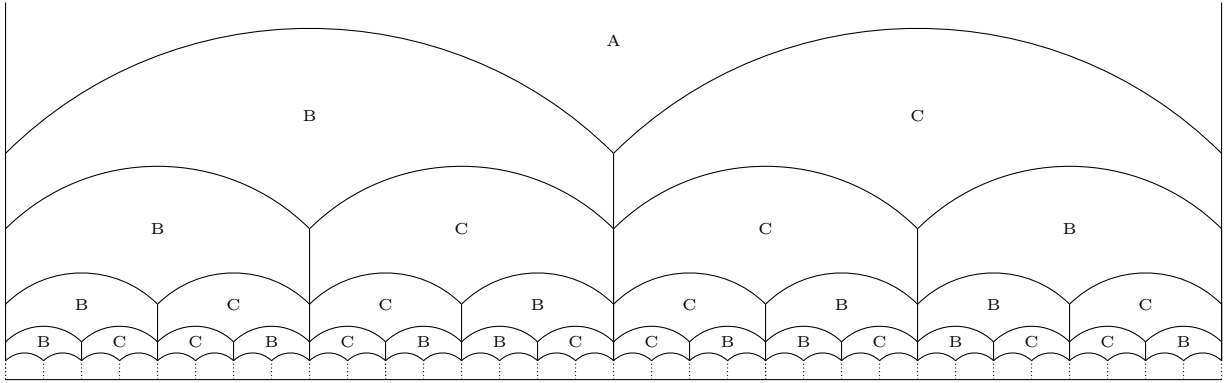
\begin{figure}[ht!]
\begin{center}
\begin{tikzpicture}[scale=.5]
\draw (-22,0)--(10,0);
\draw (-2,.5)--(-2,2);
\draw (-6,.5)--(-6,4);
\draw (-10,.5)--(-10,2);
\draw (-14,.5)--(-14,4);
\draw (2,.5)--(2,4);
\draw (10,.5)--(10,4);
\draw (10,.5)--(10,2);
\draw (6,.5)--(6,2);

\draw (0,.5)--(0,1);
\draw (4,.5)--(4,1);
\draw (8,.5)--(8,1);
\draw (-4,.5)--(-4,1);
\draw (-8,.5)--(-8,1);
\draw (-12,.5)--(-12,1);

\draw (2,1) arc(45:135:1.414);
\draw (4,1) arc(45:135:1.414);
\draw (6,1) arc(45:135:1.414);
\draw (8,1) arc(45:135:1.414);
\draw (10,1) arc(45:135:1.414);
\draw (0,1) arc(45:135:1.414);
\draw (-2,1) arc(45:135:1.414);
\draw (-4,1) arc(45:135:1.414);
\draw (-6,1) arc(45:135:1.414);
\draw (-8,1) arc(45:135:1.414);
\draw (-10,1) arc(45:135:1.414);
\draw (-12,1) arc(45:135:1.414);
\draw (-14,1) arc(45:135:1.414);
\draw (-16,1) arc(45:135:1.414);
\draw (-18,1) arc(45:135:1.414);
\draw (-20,1) arc(45:135:1.414);

\draw (1,0.5) arc(45:135:.707);
\draw (2,0.5) arc(45:135:.707);
\draw (3,0.5) arc(45:135:.707);
\draw (4,0.5) arc(45:135:.707);
\draw (5,0.5) arc(45:135:.707);
\draw (6,0.5) arc(45:135:.707);
\draw (7,0.5) arc(45:135:.707);
\draw (8,0.5) arc(45:135:.707);
\draw (9,0.5) arc(45:135:.707);
\draw (10,0.5) arc(45:135:.707);
\draw (0,0.5) arc(45:135:.707);
\draw (-1,0.5) arc(45:135:.707);
\draw (-2,0.5) arc(45:135:.707);
\draw (-3,0.5) arc(45:135:.707);
\draw (-4,0.5) arc(45:135:.707);
\draw (-5,0.5) arc(45:135:.707);
\draw (-6,.5) arc(45:135:.707);
\draw (-7,0.5) arc(45:135:.707);
\draw (-8,0.5) arc(45:135:.707);
\draw (-9,0.5) arc(45:135:.707);
\draw (-10,0.5) arc(45:135:.707);
\draw (-12,0.5) arc(45:135:.707);
\draw (-11,0.5) arc(45:135:.707);
\draw (-13,0.5) arc(45:135:.707);
\draw (-15,0.5) arc(45:135:.707);
\draw (-17,0.5) arc(45:135:.707);
\draw (-14,0.5) arc(45:135:.707);
\draw (-16,0.5) arc(45:135:.707);
\draw (-18,0.5) arc(45:135:.707);
\draw (-19,0.5) arc(45:135:.707);
\draw (-20,0.5) arc(45:135:.707);
\draw (-21,0.5) arc(45:135:.707);

\draw (-10,2) arc(45:135:2.828);
\draw (2,2) arc(45:135:2.828);
\draw (-2,2) arc(45:135:2.828);
\draw (6,2) arc(45:135:2.828);
\draw (10,2) arc(45:135:2.828);
\draw (-6,2) arc(45:135:2.828);
\draw (-14,2) arc(45:135:2.828);
\draw (-18,2) arc(45:135:2.828);

\draw[densely dotted] (-22,0)--(-22,.5);
\draw[densely dotted] (-20,0)--(-20,.5);
\draw[densely dotted] (-18,0)--(-18,.5);
\draw[densely dotted] (-16,0)--(-16,.5);
\draw[densely dotted] (-14,0)--(-14,.5);
\draw[densely dotted] (-12,0)--(-12,.5);
\draw[densely dotted] (-10,0)--(-10,.5);
\draw[densely dotted] (-8,0)--(-8,.5);
\draw[densely dotted] (-6,0)--(-6,.5);
\draw[densely dotted] (-4,0)--(-4,.5);
\draw[densely dotted] (-2,0)--(-2,.5);
\draw[densely dotted] (0,0)--(0,.5);
\draw[densely dotted] (4,0)--(4,.5);
\draw[densely dotted] (6,0)--(6,.5);
\draw[densely dotted] (8,0)--(8,.5);
\draw[densely dotted] (2,0)--(2,.5);
\draw[densely dotted] (10,0)--(10,.5);
\draw[densely dotted] (-2,0)--(-2,.5);
\draw[densely dotted] (-13,0)--(-13,.5);
\draw[densely dotted] (-11,0)--(-11,.5);
\draw[densely dotted] (-9,0)--(-9,.5);
\draw[densely dotted] (-7,0)--(-7,.5);
\draw[densely dotted] (-5,0)--(-5,.5);
\draw[densely dotted] (-3,0)--(-3,.5);
\draw[densely dotted] (-1,0)--(-1,.5);
\draw[densely dotted] (1,0)--(1,.5);
\draw[densely dotted] (3,0)--(3,.5);
\draw[densely dotted] (5,0)--(5,.5);
\draw[densely dotted] (7,0)--(7,.5);
\draw[densely dotted] (9,0)--(9,.5);

\draw (10,6) arc(45:135:11.313);
\draw (-6,6) arc(45:135:11.313);

\draw (10,.5)--(10,10);
\draw (-6,6)--(-6,.5);
\draw (-22,10)--(-22,.5);
\draw[densely dotted] (-22,0)--(-22,.5);

\draw (-18,.5)--(-18,2);
\draw (-16,.5)--(-16,1);
\draw (-20,.5)--(-20,1);
\draw[densely dotted] (-17,0)--(-17,.5);
\draw[densely dotted] (-19,0)--(-19,.5);
\draw[densely dotted] (-21,0)--(-21,.5);
\draw[densely dotted] (-15,0)--(-15,.5);
\draw[densely dotted] (-12,0)--(-12,.5);

\draw (2,4) arc(45:135:5.656);
\draw (10,4) arc(45:135:5.656);
\draw (-6,4) arc(45:135:5.656);
\draw (-14,4) arc(45:135:5.656);

\draw (-6,9) node{\tiny{A}};

\draw (2,7) node{\tiny{C}};
\draw (-14,7) node{\tiny{B}};

\draw (-18,4) node{\tiny{B}};
\draw (-2,4) node{\tiny{C}};
\draw (6,4) node{\tiny{B}};
\draw (-10,4) node{\tiny{C}};

\draw (-12,2) node{\tiny{C}};
\draw (-8,2) node{\tiny{B}};
\draw (-4,2) node{\tiny{C}};
\draw (0,2) node{\tiny{B}};
\draw (-20,2) node{\tiny{B}};
\draw (-16,2) node{\tiny{C}};
\draw (4,2) node{\tiny{B}};
\draw (8,2) node{\tiny{C}};

\draw (-21,1) node{\tiny{B}};
\draw (-19,1) node{\tiny{C}};
\draw (-15,1) node{\tiny{B}};
\draw (-17,1) node{\tiny{C}};
\draw (-11,1) node{\tiny{B}};
\draw (-13,1) node{\tiny{C}};
\draw (-9,1) node{\tiny{B}};
\draw (-7,1) node{\tiny{C}};

\draw (-5,1) node{\tiny{C}};
\draw (-3,1) node{\tiny{B}};
\draw (-1,1) node{\tiny{B}};
\draw (1,1) node{\tiny{C}};
\draw (3,1) node{\tiny{B}};
\draw (5,1) node{\tiny{C}};
\draw (7,1) node{\tiny{C}};
\draw (9,1) node{\tiny{B}};

\end{tikzpicture}
\end{center}
\caption{Geometric realization of the complex obtained by iteration of 
the topological substitution $\alpha$}\label{fig:penta-tiling}
\end{figure}

For all $k\in\N$, the cores of $\alpha^k(B)$ and $\alpha^k(C)$
are empty. For $k\geq 3$, the core of $\alpha^k(A)$ is nonempty,
but there are only faces of type $B$ and $C$ in it. 
Since there is a face of type $A$ in $\alpha(A)$
(even if it is not in the core of $\alpha(A)$),
we can obtain a 2-complex as the increasing union of the $\alpha^k(A)$
as in Section~\ref{sec:inflation}:
$$\alpha^\infty(A)=\bigcup_{k=0}^\infty\alpha^k(A).$$
But $\alpha^\infty(A)$ is not homeomorphic to the plane.

Nevertheless, $\alpha^\infty(A)$ can be realized as a tiling of a convex subset of 
the hyperbolic plane $\Hyp^2$, as explained in the following.
For now, we use the half-plane model for $\Hyp^2$:
$$\Hyp^2=\{z\in\C: \im(z)>0\}.$$
We consider the points 
$t_1=1+2i$, $t_2=2i$, $t_3=i$, $t_4=1/2+i$, $t_5=1+i$ in $\Hyp^2$, 
and the pentagon $T$ obtained by joining $t_j$ to $t_{j+1}$ 
($j\in\{1,\dots,5\}$ modulo 5) by geodesic arcs.
Let $\Gamma$ be the subgroup of $\isom(\Hyp^2)$ generated by the 
maps $z\mapsto z/2$ and $z\mapsto z+1$.
The hyperbolic plane $\Hyp^2$ is tiled by the set 
$\{g T:g\in\Gamma\}$.
This example of tiling is often attributed to Penrose
and has been fruitfully used by several authors;
for instance by Goodman-Strauss \cite{Good.05}, or Petite \cite{Pet.06}.

Let $Q\subseteq \mathbb{H}^2$ be the part of $\mathbb{H}^2$ between the line 
$Re(z)=0, Re(z)=1$ and under the geodesic joining $t_1, t_2$. 
We remark that the intersection of a tile $T'$ of the preceding tiling 
and $Q$ is either the empty set or the tile $T'$: the tiling restricts to $Q$.
It can be easily checked, using the configuration graph of vertices $\mathcal{C}(\alpha)$ (see Figure \ref{fig:hpara-val}), that the complex induced by the tiling of $Q$ is isomorphic 
(as a $2$-complex) to the unlabelled complex $\alpha^{\infty}(A)$. 
Hence $\alpha^\infty(A)$ can be realized as this tiling of $Q$ by taking $T$ 
as the geometric realization of each tile $A$, $B$ or $C$  
(consistently with respect to the numbering of the vertices)
-- see Figure~\ref{fig:penta-tiling}.

One can easily derive from $Q$ a tiling of the whole plane $\Hyp^2$
in the following way. Consider a pentagon in $Q$ labelled by $B$. 
The tiling restricts to the part of the hyperbolic plane below this pentagon
(i.e. the part containing this pentagon, and bounded by the segment $b_{12}$ 
and the half lines containing $b_{23}$ and $b_{15}$):
we call this set a band.
For all pentagons in $Q$ labelled by $B$, the corresponding bands
are tiled in the same way: there exists an element $g\in\isom(\Hyp^2)$
which send isometrically one band to the other one, and any tile of the band
on a tile with the same label.
A given band $Q_1$ contains, in its interior, a tile labelled by $B$,
and thus the corresponding band, that we denote by $Q_0$.
Considering that $Q_1$ plays the role of $Q_0$, and so on,
we can produce a chain:
$Q_0\subseteq Q_{1}\subseteq \dots \subseteq Q_{k-1}\subseteq Q_k$
for any $k\geq 0$. The union $\bigcup_{k\geq 0} Q_k$ is a tiling of $\Hyp^2$.

\subsection{A substitutive tiling of $\mathbb{H}^2$: the substitution $\beta$}
\label{subsec:subst-pav}

For the tilings of $\mathbb{H}^2$ obtained in the previous example the point $\infty$ plays a particular role. In some sense we can think there is a source at $\infty$ which generates the tiling. To obtain a substitutive hyperbolic tiling, the idea is to move the source from $\infty$ to a point inside $\Hyp^2$.

\begin{figure}[ht!]
\begin{center}
\includegraphics[width= 8cm]{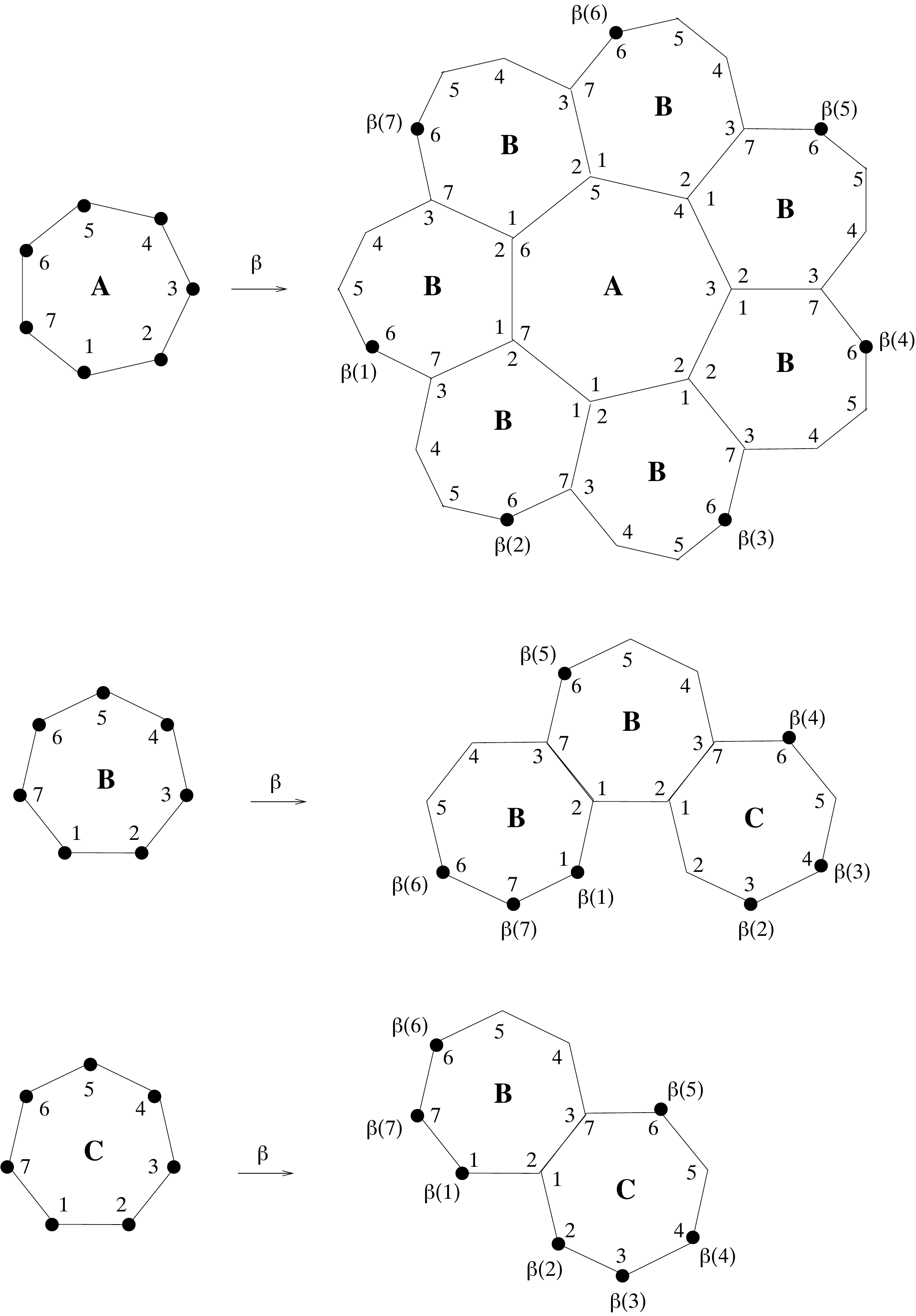}

\end{center}
\caption{Example of a non-primitive topological substitution $\beta$}\label{fig:hepta}
\end{figure}

For that, we consider the pre-substitution $\beta$ defined on Figure~\ref{fig:hepta}.
The compatibility of $\beta$ is checked by computing the heredity
graph of edges of $\beta$: see Figure~\ref{fig:hepta:edge}. 

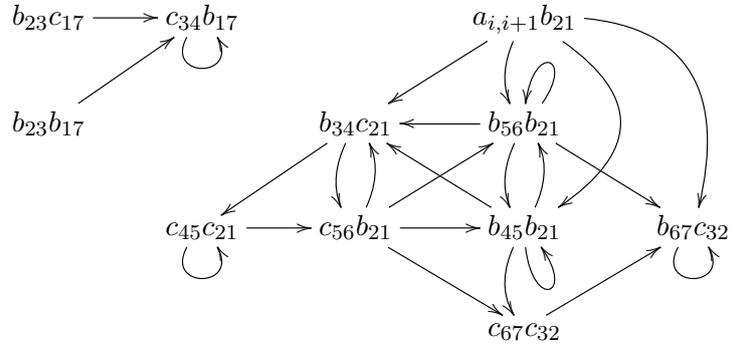
\begin{figure}[ht!]
\begin{center}
$\xymatrix{
b_{23}c_{17}\ar[r] & c_{34}b_{17}\ar@(dl,dr)[] & & a_{i,i+1}b_{21}\ar[dl]\ar@(ur,ul)[d]\ar@/^3pc/[ddr]\ar@/^3pc/[dd] & \\
b_{23}b_{17}\ar[ur] & & b_{34}c_{21}\ar[dl]\ar@(ur,ul)[d] & b_{56}b_{21}\ar[dr]\ar@(ur,u)[]\ar@(ur,ul)[d]\ar[l] &\\
& c_{45}c_{21}\ar[r]\ar@(dl,dr)[] & c_{56}b_{21}\ar[r]\ar[ur]\ar[dr]\ar@(dl,dr)[u] & b_{45}b_{21}\ar@(d,dr)[]\ar@(ur,ul)[d]\ar@(dl,dr)[u]\ar[lu] & b_{67}c_{32}\ar@(dl,dr)[]\\
 & & & c_{67}c_{32}\ar[ur] &
}$
\caption{$\mathcal{E}(\beta)$ for the heptagonal substitution}\label{fig:hepta:edge}
\end{center}
\end{figure}

\begin{figure}[ht!]
\begin{center}
$ \xymatrix{ 
a_ib_1b_2\ar[d]&b_5b_1b_2\ar[dl]\\
b_6b_1c_3\ar@(dl,dr)&c_6b_1c_3\ar[l]&c_5c_1b_2\ar[l]\\
&b_4c_1b_2\ar[u]
}$
\hspace{0.3cm}
$\xymatrix{
b_3b_7c_2\ar[d]\\
c_4b_7c_2\ar@(dl,dr)&b_3c_7c_2\ar[l]
}$
\caption{$\mathcal{C}(\beta)$ for the heptagonal tiling}\label{fig:val3}
\end{center}
\end{figure}
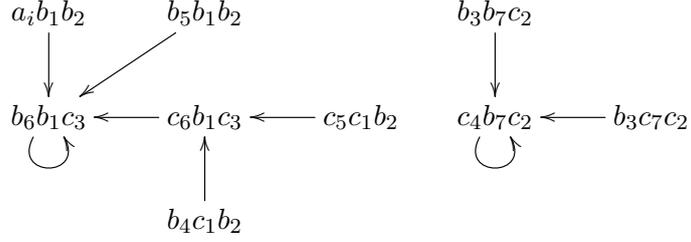

The core of $\beta(A)$ consists of a tile of type $A$. 
The pre-substitution $\beta$ is thus a substitution, 
and we can consider the complex $\beta^\infty(A)$.
This complex has the following properties: 
all the faces are heptagons, and the configuration graph of vertices  $\mathcal{C}(\beta)$ (see Figure \ref{fig:val3}) shows that each vertex is of valence $3$. 
Thus the unlabelled complex can be geometrized as a tiling 
by regular heptagons with angle $\frac{2\pi}{3}$, which proves:

\begin{thm}\label{thm:nprim}
There exists a (non-primitive) substitutive tiling of the 
hyperbolic plane $\Hyp^2$. The complex associated to this tiling is obtained by inflation from the topological substitution $\beta$.
\end{thm}

\begin{figure}[ht!]
\begin{center}
\includegraphics[width=9cm]{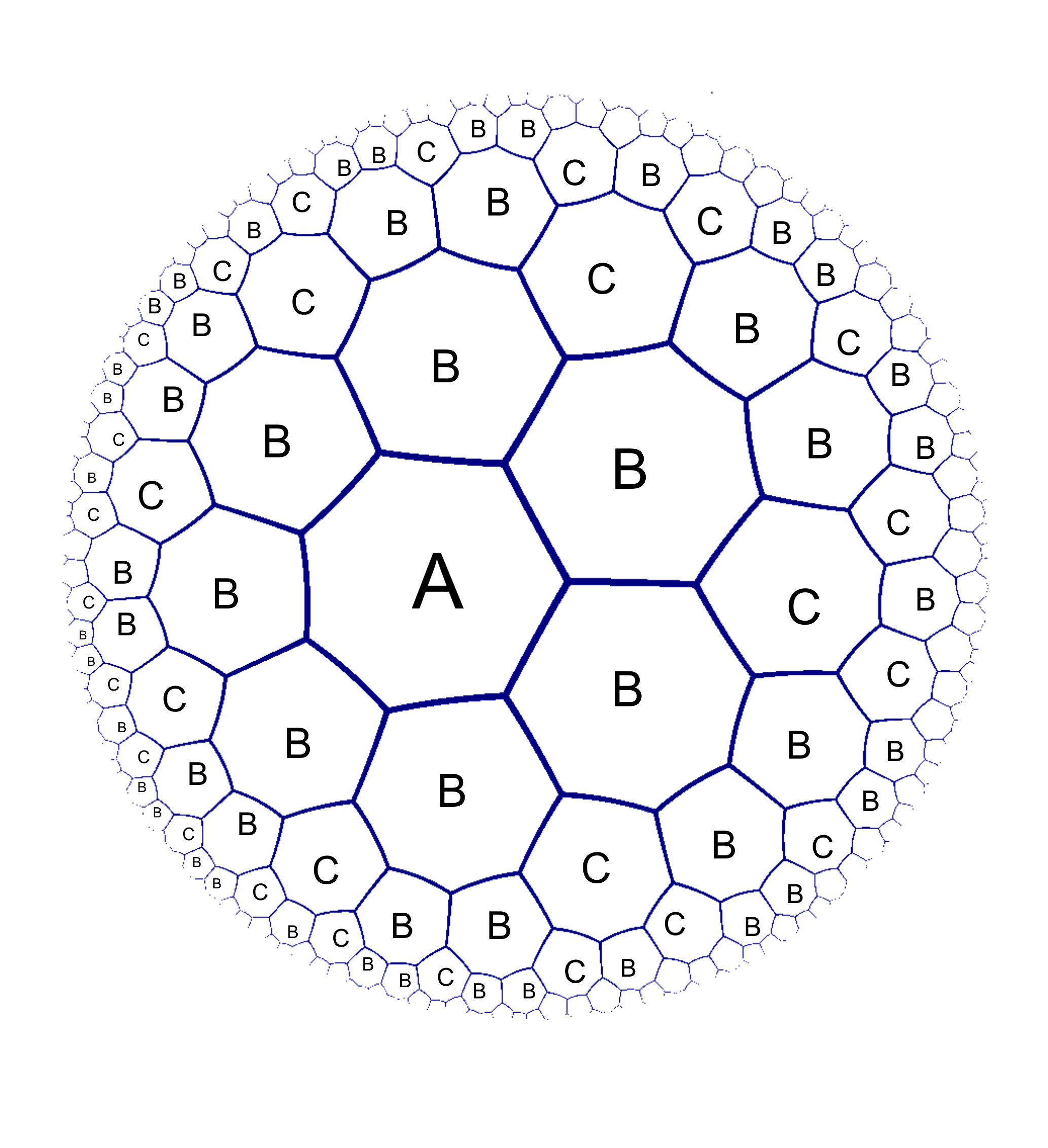}
\end{center}
\caption{Geometric realization of the complex obtained by iteration of the topological substitution $\beta$}\label{fig:hepta-tiling}
\end{figure}

This tiling is pictured in Figure~\ref{fig:hepta-tiling}.
The iteration of the  substitution on the heptagons $B, C$ 
give rise to complexes with empty core. 
Roughly speaking one can describe what happens as follows. The tiling of 
$\Hyp^2$ by heptagons is a grid and we are going to label each 
heptagon by $A, B$ or $C$.
First, we chose an heptagon, and label it by $A$. 
Applying $\beta$ will add a label $B$ on each heptagon
surrounding the one labelled by $A$.
Then, each iteration of $\beta$ is going to add a label $B$ or $C$ on each tile of the annulus surrounding the previous picture. 

\section{Appendix}
We include here a proof that the map $d$ defined in Section~\ref{sec:dynamics} is a distance. We use the notations of Section~\ref{sec:dynamics}. We mainly follow \cite[Lemma 2.7]{Robin.04} to prove the triangle inequality.

\begin{lem}\label{lem:pseudonorme}
Let $g,h\in\isom(M)$, $\alpha,\beta\geq 0$. Then
\begin{enumerate}[(i)]
\item \label{-1} $||g^{-1}||_\alpha=||g||_\alpha$
\item \label{croiss} $\alpha\leq \beta \;\Rightarrow\; ||g||_\alpha\leq ||g||_\beta$
\item \label{inclus}  $||g||_\alpha\leq \beta \;\Rightarrow\; gB_\alpha\supseteq B_{\alpha-\beta}$
\item \label{triang} $||hg||_\alpha\leq ||h||_\alpha+||g||_\alpha$
\end{enumerate}
\end{lem}
\begin{proof}
Properties (\ref{-1}) and (\ref{croiss}) follow immediately from the definition.
Suppose that $\alpha\geq \beta$ and $||g||_\alpha\leq \beta$, and consider $Q\in B_{\alpha-\beta}$ (if $\alpha<\beta$ Property $(iii)$ is trivial). 
Then 
$$d(gQ,O)\leq d(gQ,Q)+d(Q,O)\leq 
||g||_{\alpha-\beta}+\alpha-\beta
\leq ||g||_{\alpha}+\alpha-\beta
\leq \alpha,$$
which proves $gB_{\alpha-\beta}\subseteq B_{\alpha}$, 
and then, by property (\ref{-1}), the property (\ref{inclus}).
Property (\ref{triang}) follows from
$$d(hgQ,Q)\leq d(hgQ,hQ)+d(hQ,Q)
=d(gQ,Q)+d(hQ,Q)$$
for any point $Q\in B_\alpha$.
\end{proof}

\begin{prop}\label{lem:distance}
The map $d$ is a distance on $\Sigma_{\mathcal{P}}$.
\end{prop}

\begin{proof}
For $x,y\in \Sigma_{\mathcal{P}}$, clearly $d(x,y)=d(y,x)$
and $d(x,y)=0$ implies $x=y$.
We consider three tilings $x,y,z\in \Sigma_{\mathcal{P}}$: 
we are going to prove that $d(x,z)\leq d(x,y)+d(y,z)$.
We can assume that $d(x,y)=a\leq d(y,z)=b$ and $a+b<1$. 
Let $\varepsilon>0$ such that $\varepsilon<\frac{1-(a+b)}{2}$.
We take 
$\alpha=a+\varepsilon,\; \beta=b+\varepsilon,\; \gamma=\alpha+\beta$
(and thus $\gamma<1$). 
There exist
$$p\in B_{1/\alpha}[x], \;q\in B_{1/\alpha}[y],\; q'\in B_{1/\beta}[y], \;
r\in B_{1/\beta}[z], \; g,h\in \Gamma$$ such that
$$gp =q, \; hq'=r, \; ||g||_{1/\alpha}\leq \alpha, \; ||h||_{1/\beta}\leq \beta.$$
Now we define $q_0=q\cap q'$ and $p_0=g^{-1}q_0, \; r_0=hq_0$.
Thus $hgp_0=r_0$. 

By properties (\ref{-1}) and (\ref{croiss}) of Lemma \ref{lem:pseudonorme},
we obtain that: 
$$||g^{-1}||_{1/\beta}= ||g||_{1/\beta} \leq ||g||_{1/\alpha}\leq\alpha.$$
We deduce from property (\ref{inclus}) that 
$g^{-1}B_{1/\beta}\supset B_{1/\beta-\alpha}$.
Moreover, since $\beta,\gamma<1$, we remark that 
$1/\gamma\leq 1/\beta-\alpha.$
Thus $g^{-1}B_{1/\beta}\supset B_{1/\gamma}.$
Since $q_0\in B_{1/\beta}[y]$, we deduce that 
$p_0=g^{-1}q_0\in B_{1/\gamma}[x].$

Finally, using property (\ref{triang}) and (\ref{croiss}) of Lemma \ref{lem:pseudonorme}, 
we obtain that:
$$||hg||_{1/\gamma} \leq ||h||_{1/\gamma} + ||g||_{1/\gamma}
\leq ||h||_{1/\beta} + ||g||_{1/\alpha}
\leq\gamma.$$
This proves that $d(x,z)\leq \gamma\leq d(x,y)+d(y,z)+2\varepsilon$ for 
arbitrarily small $\varepsilon$, and the triangle inequality follows.
\end{proof}
We note that, if $M=\E^2$ and $\Gamma$ is the group of translations of $\E^2$,
the definition of the distance $d$ given below coincides precisely
with the classical one. 

\medskip

We include below a proof, in our context, of Gottschalk's Theorem (Proposition~\ref{prop:repet}). 

\begin{proof}
A tiling $y$ is in $\Omega(x)$ if and only if there exists a sequence 
of tilings $g_n x$ ($g_n\in\Gamma$) converging to $y$. 
It means that for every integer $n$ there exists a patch $Y_n$ of $y$ and a patch $P_n$ of $g_n x$, 
each of whose supports contains a ball of radius $n$ centered at $0$,
and an isometry $g\in \Gamma$ with $||g||\leq 1/n$ such that $Y_n=gP_n$. 

(i) If $y\in \Omega(x)$, any protopatch of $y$ occurs in the support of some $Y_n=gP_n$, thus occurs in $x$. Hence $\mathcal{L}(y)\subseteq \mathcal{L}(x)$. Conversely let $B$ be the ball of radius $n$ centered in $0$. Let $Y_n$ be a patch of $y$, the support of which contains $B$.  By hypothesis the protopatch defined by $Y_n$ is a protopatch of $x$. Thus there exists a patch $P_n$ of $x$ and an element $g_n\in \Gamma$ such that $g_nY_n=P_n$. The sequence of tilings $(g_{n}^{-1}x)$ converges to $y$ by construction.

(ii) Let $\hat{P}$ be a protopatch of $\mathcal{L}(x)$. The repetitivity of $x$ implies there exists some $R>0$ such that the protopatch $\hat{P}$ occurs in every ball of radius bigger than $R$. Thus for $n\geq R$, $\hat{P}$ occurs in the support of $P_n=g^{-1}Y_n$, and hence also occurs in $y$. This proves that $\mathcal{L}(x)=\mathcal{L}(y)$. It follows that $y$ is a repetitive tiling.

(iii) Let $y$ be in $\Omega(x)$. Property $(ii)$ implies that $\mathcal{L}(y)=\mathcal{L}(x)$, and thus by Property $(i)$, $x\in \Omega(y)$. Hence $(\Omega(x), \Gamma)$ is minimal.

(iv) Let $B_n$ be the ball of radius $n$ centered at the origin. Suppose that $x$ is not a repetitive tiling. There exists a protopatch  $\hat{P}\in \mathcal{L}(x)$ such that, for every integer $n$, there exists an element $g_n\in \Gamma$, and a patch $P_n$, the support of which contains $g_nB_n$, such that $\hat{P}$ does not occur in $P_n$. 
By compactness, we can assume that the sequence $g_n^{-1}x$ converges. Let $y$ be this limit tiling, then $\hat{P}$ is not a protopatch of $y$. Thus $\Omega(y)\subsetneq \Omega(x)$, and $(\Omega(x), \Gamma)$ is not a minimal system.

\end{proof}


\bibliographystyle{alpha}
\bibliography{biblio-pav}

\end{document}